\documentclass[a4paper,12pt,final]{article}
\usepackage[scr=boondoxo]{mathalfa}
\usepackage{amsmath}
\usepackage{array}
\usepackage{cancel}
\usepackage{amssymb}
\usepackage{amsthm}
\usepackage{mathtools}
\usepackage[dvipsnames]{xcolor}
\usepackage{mathabx}
\usepackage{accents}
\usepackage[all]{xy}
\usepackage{tikz}
\usetikzlibrary{matrix,positioning,arrows}
\usepackage[inline]{enumitem}
\usepackage[utf8]{inputenc}
\usepackage{textcomp}
\usepackage[unicode]{hyperref}
\usepackage{cleveref}
\usepackage{thmtools}

\newcommand*{\doi}[1]{{\sc doi: }\href{http://dx.doi.org/#1}{#1}}

\hypersetup{
    colorlinks,
    linkcolor={blue!80!black},
    citecolor={blue!80!black},
    urlcolor={blue!80!black}
}

\DeclareMathAccent{\wtilde}{\mathord}{largesymbols}{"65}

\declaretheorem[name=Theorem,refname={theorem,theorems},Refname={Theorem,Theorems}]{theorem}
\declaretheorem[name=Definition,refname={definition,definitions},Refname={Definition,Definitions},sibling=theorem]{definition}
\declaretheorem[name=Proposition,refname={proposition,propositions},Refname={Proposition,Propositions},sibling=theorem]{proposition}
\declaretheorem[name=Corollary,refname={corollary,corollaries},Refname={Corollary,Corollaries},sibling=theorem]{corollary}
\declaretheorem[name=Lemma,refname={lemma,lemmas},Refname={Lemma,Lemmas},sibling=theorem]{lemma}

\newcommand{\coleq}{\coloneqq}
\newcommand{\csn}{\mathrm{csn}}
\newcommand{\ud}{\mathrm{d}}
\newcommand{\pd}{\partial}
\newcommand{\norm}[1]{\left\lVert#1\right\rVert}
\newcommand{\abso}[1]{\left|#1\right|}
\newcommand{\csub}{\subset \subset}

\newcommand{\bN}{\mathbb{N}}

\newcommand{\bR}{\mathbb{R}}
\newcommand{\bC}{\mathbb{C}}
\newcommand{\cD}{\mathcal{D}}

\newcommand{\cB}{\mathcal{B}}

\newcommand{\cN}{\mathcal{N}}
\newcommand{\cU}{\mathcal{U}}
\newcommand{\cF}{\mathcal{F}}
\newcommand{\cL}{\mathcal{L}}
\newcommand{\cE}{\mathcal{E}}
\newcommand{\cS}{\mathcal{S}}
\newcommand{\basic}{\cE}
\newcommand{\moderate}{\cE_M}
\newcommand{\negligible}{\mathcal{N}}
\newcommand{\quotient}{\mathcal{G}}

\newcommand{\bPhi}{{\boldsymbol\Phi}}

\newcommand{\bvarphi}{{\boldsymbol\varphi}}
\newcommand{\btheta}{{\boldsymbol\theta}}
\newcommand{\bpsi}{{\boldsymbol\psi}}

\newcommand{\e}{\varepsilon}
\newcommand{\SO}{\mathrm{SO}}
\newcommand{\D}{\mathrm{D}}
\newcommand{\SK}{\mathrm{SK}}
\DeclareMathOperator{\supp}{supp}

\DeclareMathOperator{\pr}{pr}

\DeclareMathOperator{\id}{id}

\newcommand{\cG}{\mathcal{G}}
\title{Full and special Colombeau algebras}
\date{November 16, 2016}
\author{M.~Grosser\footnote{Universität Wien, Oskar-Morgenstern-Platz 1, 1090 Vienna, Austria.\newline email: \href{mailto:michael.grosser@univie.ac.at}{michael.grosser@univie.ac.at}.}, E.~A.~Nigsch\footnote{Wolfgang-Pauli-Institut, Oskar-Morgenstern-Platz 1, 1090 Vienna, Austria.\newline email: \href{mailto:eduard.nigsch@univie.ac.at}{eduard.nigsch@univie.ac.at}.}}

\newcommand{\unterstrich}{\mathunderscore\kern-1pt\mathunderscore\,}

\begin{document}

\maketitle

\begin{abstract}
We introduce full diffeomorphism-invariant Colombeau algebras with added $\e$-dependence in the basic space. This unites the full and special settings of the theory into one single framework. Using locality conditions we find the appropriate definition of point values in full Colombeau algebras and show that special generalized points suffice to characterize elements of full Colombeau algebras. Moreover, we specify sufficient conditions for the sheaf property to hold and give a definition of the sharp topology in this framework.
\end{abstract}

\textbf{Keywords: Colombeau algebra, generalized function, full, special, sharp topology, point values, locality, sheaf property}

\textbf{MSC2010 classification: 46F30}

\section{Introduction}\label{sec_intro}

Colombeau algebras are algebras of nonlinear generalized functions introduced by J.~F.~Colombeau \cite{ColNew,ColElem} in order to solve the problem of multiplication of distributions and hence circumvent the impossibility result of L.~Schwartz \cite{Schwartz}.

Contrary to the theory of distributions, where there is a generally accepted ``canonical'' definition of the space $\cD'(\Omega)$ of distributions on an open subset $\Omega \subseteq \bR^n$ on which any further studies are based, the theory of Colombeau algebras is -- at least until now -- open-ended in the sense that there is not one single ``canonical'' Colombeau algebra. Rather, one finds a family of constructions centered around the same principles but often only loosely related to each other. We give a short overview of the main representatives of this family in \Autoref{sec_history}. While in principle it is beneficial that one can choose a construction which is adapted to a particular problem at hand, the drawback of this situation consists in the fact that one often cannot formally transfer results obtained for one Colombeau algebra to a second; at best, one can hope to emulate them in the second setting (see for example \cite{Steinbauer:98,geodesics} for such a pair).

In this article we will for the first time combine several formative elements of different types of Colombeau algebras into one single conceptual frame:
\begin{itemize}
 \item Representatives of generalized functions being \emph{nets of smooth functions} indexed by $\e \in I \coleq (0,1]$; this can be seen as an extension of the sequential approach to distributions (see \cite{zbMATH03421400,zbMATH03049411} and \cite[Section 1.2.1]{GKOS}). This is the key idea of special (sometimes called simplified) Colombeau algebras.
 \item Representatives of generalized functions being \emph{smooth mappings} defined on the space $\cD(\Omega)$ of test functions; this is the original approach of Colombeau \cite{ColNew} based on calculus on infinite-dimensional locally convex spaces as in \cite{0564.46031}, nowadays customarily replaced by the convenient calculus of \cite{KM}. This provides the core concept for full Colombeau algebras.
 \item Employing \emph{smoothing operators} and the corresponding approximation properties for the formulation of the quotient construction; this idea goes back to \cite{2010arXiv1003.3341D,1204.58021} and was further developed in \cite{papernew}.
 \item Introducing various \emph{locality properties} for elements of the basic space in order for the quotient algebra to become a sheaf. This aspect having been manageable in an easy and natural way for the early versions of Colombeau algebras, it has not been dealt with systematically so far. However, this issue requires closer attention once smoothing operators get involved.
\end{itemize}
This will lead us to a natural construction of a Colombeau algebra $\quotient$ which will contain as subalgebras (at least on the level of the basic spaces) each of the following:
\begin{enumerate}[label=(\roman*)]
 \item $\cG^s$ (the special algebra \cite{0646.76007}),
 \item $\cG^o$ (Colombeau's original algebra \cite{ColNew}),
 \item $\cG^d$ (the diffeomorphism-invariant algebra \cite{found}),
 \item $\cG^f$ (the algebra obtained via the functional analytic approach \cite{papernew}).
 \end{enumerate}

We will exemplify the usefulness of $\quotient$ from several points of view: by discussing the sheaf property (\Autoref{sec_sheaf}), point values (\Autoref{sec_pv}) and the sharp topology (\Autoref{sec_top}) on it. The fact that these concepts naturally specialize to the subalgebras listed above underscores the universality of $\quotient$.

\section{History}\label{sec_history}

The roots of algebras of nonlinear generalized functions trace back as far as to the constructions of H.~König \cite{zbMATH03108760}. After subsequent contributions by E.~E.~Rosinger and J.~Egorov, J.~F.~Colombeau presented his construction of an algebra $\cG^o$ of nonlinear generalized functions in \cite{ColNew}. Its basic space on an open subset $\Omega \subseteq \bR^n$ is given by
\[ \cE^o ( \Omega ) \coleq C^\infty ( \cD(\Omega)) \]
which obviously allows for a canonical embedding of distributions. We understand smoothness on a locally convex space $E$ in the sense of convenient calculus \cite{KM}. Moreover, $C^\infty(E)$ denotes the space of smooth complex-valued functions on $E$. In order to get rid of the inherent technical difficulties resulting from the dependence on calculus on infinite-dimensional locally convex spaces (which is necessary in order to talk of smooth functions on $\cD(\Omega)$), two alternatives to $\cG^o$ have been developed. 

First, by incorporating the linear structure of $\bR^n$ into the testing procedure it is possible to base the construction of a Colombeau algebra $\cG^e(\Omega)$ (the ``elementary'' full Colombeau algebra \cite{ColElem}) on the basic space
\[ \cE^e ( \Omega ) \coleq \{ R \colon U ( \Omega ) \to \bC \ |\ R(\varphi, .) \textrm{ is smooth for all }\varphi \in \pr_1(U(\Omega))\} \]
where $U ( \Omega )$ is defined as $\{ (\varphi, x) \in \cD(\bR^n) \times \Omega\ |\ \supp \varphi + x \subseteq \Omega \}$ and $\pr_1$ denotes projection on the first component.

Second, by basing the whole construction on the sequential approach to distributions one can view distributions and, consequently, generalized functions as nets of smooth functions on $\Omega$ and define their multiplication componentwise (w.r.t. $\e$). The usual Colombeau-type quotient construction then yields an algebra $\cG^s$ which is a suitable quotient of the basic space
\[ \cE^s (\Omega) \coleq C^\infty(\Omega)^I. \]

The diffeomorphism-invariant full algebra $\cG^d$ 
was eventually obtained in \cite{found}. Its basic space
\[ \cE^d(\Omega) \coleq C^\infty(\cD(\Omega) \times \Omega) \cong C^\infty(\cD(\Omega), C^\infty(\Omega)) \]
already appears in the groundbreaking work of J.~Jel\'inek \cite{Jelinek}; the isomorphism here is given by the exponential law \cite[3.12, p.~30]{KM}. While the form of $\cE^d(\Omega)$ emerges naturally and quickly, the question of the right testing procedure giving a diffeomorphism-invariant algebra exhibits considerable technical difficulties. The solution of this problem was achieved in several steps, each broadening the understanding of the whole construction. Note that $\cG^d$ gives rise to a corresponding Colombeau algebra $\hat\cG$ on manifolds \cite{global}.

One of the key lessons to be learned from the development of diffeomorphism invariant Colombeau algebras and pointed out already in \cite{Jelinek} came as the insight that smooth dependence of $R$ on its parameter $\varphi$ is an indispensable ingredient for the theory. Thus, in a sense, Colombeau's maneuver of eliminating infinite-dimensional calculus has to be dispensed with.

Each of the Colombeau algebras considered above relies on its own smoothing procedure for the embedding of distributions. In view of this not too satisfactory state of affairs, a need for unification was felt at many occasions. Surprisingly enough, the decisive idea in this respect originated from the development of a geometric theory of generalized functions allowing, in particular, for a covariant derivative. Rather than considering a fixed ``regularization procedure'' as in each of the algebras above, the new approach taken in \cite{bigone} is based on introducing the entire space $\cL(\cD'(\Omega), C^\infty(\Omega))$ of ``smoothing operators'' as a new parameter domain into the definition of the basic space. By $\cL ( E, F)$, we denote the space of linear continuous mappings from a locally convex space $E$ to another locally convex space $F$, endowed with the topology of bounded convergence.

For $\cG^f$, which incorporates this idea in the most simple setting (the local and scalar one), the general concept that embedding distributions amounts to regularizing them has been turned into a definition: the basic space
\[ \cE^f (\Omega ) \coleq C^\infty ( \cL ( \cD'(\Omega), C^\infty(\Omega)), C^\infty(\Omega)) \]
embodies all possible ways to regularize distributions (in a linear, continuous way). By means of the kernel theorem, the canonical isomorphism
\[ \cL(\cD'(\Omega), C^\infty(\Omega)) \cong C^\infty( \Omega, \cD(\Omega)) \]
\cite[Th\'eor\`eme 3, p.~127]{FDVV} provides the link to the previously used formalisms, as will become clear from the following sections.
Moreover, this point of view enables one to obtain diffeomorphism invariance very easily; in fact, the construction as a whole even generalizes to an algebra of nonlinear generalized tensor fields on manifolds having very desirable properties without having to hassle with the technicalities of \cite{found}.

A policy adopted at one certain point of the above development is that of \emph{separating the basic definition from testing} \cite[Section 2.10]{GKOS}. Loosely speaking, testing consists of taking as ``test object'' an $\e$-dependent path $\phi(\e)$ in the domain of a representative $R$ and analyzing the asymptotic behavior of $R(\phi(\e))$ as $\e \to 0$. In all of these constructions the appropriate choice of test objects has been the major issue. Note that in all instances except for the special algebra, neither $R$ nor the elements of its domain depend on $\e$, which is what was meant by separating the basic definition from testing. Back then, this was a necessary policy to adopt in order to succeed in the next step towards a construction of a diffeomorphism invariant algebra. Presently, it seems favorable to relax this policy and examine its implications in hindsight. Consequently, in the next section we introduce a basic space which includes $\e$-dependence both in its domain and in its codomain.

\section{Basic Spaces}\label{sec_basic}

Our construction of a Colombeau algebra on an open subset $\Omega \subseteq \bR^n$ will be based on the basic space
\begin{equation}
\begin{aligned}\label{Edef}
\basic (\Omega) \coleq C^\infty & ( \cL(\cD'(\Omega), C^\infty(\Omega))^I , C^\infty(\Omega)^I ) \\
\cong C^\infty &( C^\infty(\Omega, \cD(\Omega))^I, C^\infty(\Omega)^I).
\end{aligned}
\end{equation}
Obviously, $\basic(\Omega)$ originates from $\cE^f(\Omega)$ by (re-)introducing $\e$-dependence in both its domain and its codomain. 
As in \cite{papernew} we call
\[ \SO(\Omega) \coleq \cL(\cD'(\Omega), C^\infty(\Omega)) \]
the space of \emph{smoothing operators} and
\[ \SK(\Omega) \coleq C^\infty(\Omega, \cD(\Omega)) \]
the space of \emph{smoothing kernels} on $\Omega$. These are isomorphic as locally convex spaces by L.~Schwartz' kernel theorem, which conveniently allows one to switch between the operational viewpoint using smoothing operators on the one hand and the analytic viewpoint using smoothing kernels on the other hand. The correspondence between $\Phi \in \SO(\Omega)$ and $\vec\varphi \in \SK(\Omega)$ is given by
\begin{alignat*}{2}
 \Phi(u)(x) &= \langle u, \vec\varphi(x) \rangle \qquad && (u \in \cD'(\Omega), x \in \Omega) \\
 \vec\varphi(x)(y) &= \Phi^{t}(\delta_x)(y) = \Phi(\delta_y)(x) \qquad && (x,y \in \Omega)
\end{alignat*}
where $\Phi^t$ is the transpose of $\Phi$.

\subsection{Locality in $\cE^f$}\label{subloc}

If the full basic space $\cE^f(\Omega)$ introduced in \cite{papernew} 
is used for the quotient construction the resulting algebra will fail to have 
the sheaf property. In some sense, $\cE^f(\Omega)$ is too large since for 
certain of its elements $R$, the value $R(\vec\varphi)(x)$ depends on values 
$\vec\varphi(y)$ with $y$ far away from $x$, a feature that may be viewed as 
$R$ not respecting loci on $\Omega$.

As a remedy, so-called locality conditions were 
introduced in \cite{papernew}. Keeping only those elements of the basic space 
$\cE^f(\Omega)$ displaying favorable locality properties, the 
resulting quotient algebra in fact becomes a sheaf.

Locality conditions have the additional merit of allowing to recover 
the previously used basic spaces $\cE^o(\Omega)$ and $\cE^d(\Omega)$ as 
subspaces of $\cE^f(\Omega)$.

In the following, we will recall 
these properties for the 
case of $\cE^f(\Omega)$
and extend the study to 
$\basic(\Omega)$ in the next subsection.

As to germs of sheaves, we will employ the following notation:
given a sheaf $\cF$ of vector spaces on a 
topological space $\Omega$ and an open subset $U \subseteq \Omega$, $\tau_x 
\colon \cF(U) \to \cF_x$ denotes the map assigning to a section on $U$ its germ 
at $x$, where $\cF_x$ denotes the stalk of $\cF$ at $x$.

Note that $C^\infty(\unterstrich, \cD(\Omega))$ is a sheaf of vector
spaces on $\Omega$, hence for any $\vec\varphi \in C^\infty(\Omega, \cD(\Omega))$ the germ $\tau_x\vec\varphi$ at $x \in \Omega$ is an element of the direct limit
$\varinjlim_{U \in \cU_x} C^\infty(U, \cD(\Omega))$,
where $\cU_x$ is the filter of open neighborhoods of $x$.

Items (ii) and (iii) of the following Lemma provide reformulations of the condition of locality 
in \cite[Def.~3, p.~419]{papernew}. Item (iii), in particular, paves the way for generalizing it to the framework of the present 
article.

\begin{lemma}
 Let $R \in C^\infty ( \SK(\Omega), C^\infty(\Omega))$. Then the following are equivalent:
\begin{enumerate}[label=(\roman*)]
 \item \label{lem_equiv_germloc.1} $\forall U \subseteq \Omega$ open $\forall \vec\varphi, \vec\psi \in \SK(\Omega)$:
\begin{align*}
  \vec\varphi|_U = \vec\psi|_U &\Longrightarrow R(\vec\varphi)|_U = R(\vec\psi)|_U ; \\
\intertext{ \item $\forall x \in \Omega$ $\forall \vec\varphi, \vec\psi \in \SK(\Omega)$:}
  \tau_x\vec\varphi = \tau_x\vec\psi &\Longrightarrow \tau_x(R(\vec\varphi)) = \tau_x(R(\vec\psi)) ; \\
 \intertext{ \item\label{lem_equiv_germloc.3} $\forall x \in \Omega$ $\forall \vec\varphi, \vec\psi \in \SK(\Omega)$:}
\tau_x\vec\varphi = \tau_x\vec\psi &\Longrightarrow R(\vec\varphi)(x) = R(\vec\psi)(x).
\end{align*}
\end{enumerate}
\end{lemma}
The proof is immediate. Defining locality by condition (iii) 
(rather than by (i) as in \cite{papernew}) not only gives a 
better match with the following scheme, but also is 
the appropriate form to use for \Autoref{thm_peetre} below.

\begin{definition}\label{def_flocal}A mapping $R \in C^\infty ( \SK(\Omega), C^\infty(\Omega))$ is called
 \begin{itemize}
  \item \emph{local} if $\forall x \in \Omega$ $\forall \vec\varphi, \vec\psi \in \SK(\Omega)$:
\begin{alignat*}{2}
& \tau_x\vec\varphi = \tau_x\vec\psi && \Longrightarrow R(\vec\varphi)(x) = R(\vec\psi)(x), \\
\intertext{ \item \emph{point-local} if $\forall x \in \Omega$ $\forall \vec\varphi, \vec\psi \in \SK(\Omega)$: }
& \vec\varphi(x) = \vec\psi(x) & & \Longrightarrow R(\vec\varphi)(x) = R(\vec\psi)(x), \\
\intertext{ \item \emph{point-independent} if $\forall x,y \in \Omega$ $\forall \vec\varphi, \vec\psi \in \SK(\Omega)$:}
& \vec\varphi(x) = \vec\psi(y) & & \Longrightarrow R(\vec\varphi)(x) = R(\vec\psi)(y).
\end{alignat*}
 \end{itemize}
\end{definition}
We have the strict implications \emph{point-independent} $\Rightarrow$ 
\emph{point-local} $\Rightarrow$ \emph{local}. Choosing $\{ R \in 
\cE^f(\Omega)\ |\ R \textrm{ is local}\}$ as basic space guarantees that the 
resulting quotient algebra becomes a sheaf.

Moreover, one readily verifies that the following isomorphisms hold (see also 
\Autoref{basiciso} below):
\begin{alignat*}{3}
&\{ R \in \cE^f(\Omega)\ |\ R \textrm{ is point-local} \} && \cong C^\infty ( \cD(\Omega), C^\infty(\Omega )) && = \cE^d(\Omega),\\
&\{ R \in \cE^f(\Omega)\ |\ R \textrm{ is point-independent} \} && \cong C^\infty ( \cD(\Omega )) && = \cE^o(\Omega). \\
\intertext{As to the prominent space}
&\{ R \in \cE^f(\Omega)\ |\ R \textrm{ is local}\},
\end{alignat*}
however, there is no nice explicit representation of the above kind. Nevertheless, 
Peetre-like results can be obtained even in this nonlinear case as follows:
\begin{theorem}\label{thm_peetre} A mapping $R \in \cE^f(\Omega)$ is local in the sense of \Autoref{def_flocal} if and only if for every $x \in \Omega$ and $\vec\varphi \in \SK(\Omega)$ the value $R(\vec\varphi)(x)$ depends on the $\infty$-jet $j^\infty(\vec\varphi)(x)$ only, i.e., if for any $x \in \Omega$ and $\vec\varphi, \vec\psi \in \SK(\Omega)$ the equality $(\pd^\alpha_x \vec\varphi)(x) = (\pd^\alpha_x \vec\psi)(x)$ for all $\alpha \in \bN_0^n$ implies $R(\vec\varphi)(x) = R(\vec\psi)(x)$.
\end{theorem}
The preceding theorem as well as a stronger form of it which essentially gives locally finite order can be found in \cite{peetre}.

\subsection{Locality in $\cE$}\label{subsec_loce}

In the preceding subsection on $\cE^f(\Omega)$, it was sufficient to 
handle three locality conditions, altogether. However, to exploit the full 
conceptual potential of $\cE(\Omega)$ we will have to control a totality of 28 
locality conditions or, rather, locality types.

Therefore, it is indispensable on the one hand to set up a more 
comprehensive terminological apparatus. On the other hand, it is essential to 
develop a sufficiently clear intuitive picture of the nature of the locality 
conditions involved.

To begin with, we note that according to the exponential law  
elements of  $\cE(\Omega)$ can be regarded as complex-valued functions defined 
on triples  $(\bvarphi, x, \e) \in 
\SK(\Omega)^I \times \Omega \times I$ which are smooth in $(\bvarphi, x)$ for 
fixed $\e$. Consequently, we will use freely either of $R(\boldsymbol 
\varphi)_\e(x)$ and $R(\bvarphi,x,\e)$.

Triples $(\bvarphi, x, \e)$ may be viewed as data to be fed 
into $R\in\cE(\Omega)$, giving rise to $R(\bvarphi, x, \e)\in
C^\infty(\Omega)^I$. The key strategy of our locality approach consists in 
classifying elements $R$ of $\cE(\Omega)$ according to the degree up to which 
the full content of $(\bvarphi, x, \e)$ is actually used for determining the 
value $R(\bvarphi,x,\e)$.

For example, if $R(\bvarphi,x,\e)=S(\bvarphi_\e(x))$ 
for some $S\in C^\infty( \cD(\Omega) )$ then $R$ reacts to only a part of 
the content of $(\bvarphi,x,\e)$. We say that ``$R$ only 
depends on $\bvarphi_\e(x)$''; in fact, $\bvarphi_\e(x)$ does not carry the 
full information contained in $(\bvarphi,x,\e)$ since it is not possible to 
recover the latter from the former.

We illustrate and, at the same time, formalize this concept by 
considering the following toy model:

A function $f:\bR\times \bR\to \bR$ is said to depend 
only on $x+y$ if there exists $g:\bR\to\bR$ such that $f(x,y)=g(x+y)$. Defining 
$\ell:\bR\times \bR\to \bR$ by $\ell:(x,y)\mapsto x+y$, this property of $f$ 
amounts to the fact that it factors according to $f=g\circ \ell$ for a 
suitable $g$ or, in other words, that $f$ is the pullback of some $g$ under 
$\ell$. We express this by saying that $f$ ``is of locality type $\ell$''. More 
generally, we say that a function $f:X\to Y$ is of locality type $\ell$ with 
$\ell:X\to X_\ell$ if $f=g\circ \ell$ for a suitable $g:X_\ell \to Y$, for sets 
$X$, $X_\ell$, $Y$.

In particular, a locality type on $\SK(\Omega)^I \times \Omega \times 
I$ is but a map $\ell:\SK(\Omega)^I \times \Omega \times I\to X_\ell$. 
Intuitively, it is to be viewed as a compression procedure acting on the data 
$(\bvarphi,x,\e)$. The locality character of $R$ is reflected in the 
corresponding locality type $\ell$ through which it factors: the stronger the 
data compression effected by $\ell$, the less of the data is actually needed 
for determining the value of $R(\bvarphi,x,\e)$.

Note that the present 
concept of locality not only refers to the locus $x \in \Omega$, but equally to 
the ``loci'' $\bvarphi \in \SK(\Omega)^I$ and
$\e \in I$. For 
instance, if $R = \iota u$ is the embedded image of a distribution $u \in 
\cD'(\Omega)$ (with $\iota$ as in \eqref{embed_iota} below) then 
$R(\bvarphi)_\e(x) = \langle u,\bvarphi_\e(x)\rangle$ depends only on 
$\bvarphi_\e(x) \in \cD(\Omega)$ but not (directly) on $\e$, on $x$ or 
on $\bvarphi$. Similarly, if $R = \sigma f$ is the embedded image of a smooth function $f \in 
C^\infty(\Omega)$ (with $\sigma$ as in \eqref{embed_sigma} below) then 
$R(\bvarphi)_\e(x) = f(x)$ depends only on $x$, but not on $\e$ or on 
$\bvarphi$.

Some of the most important locality types involve germs of $\bvarphi$ at 
$x\in\Omega$. Thus we have to introduce one more piece of terminology along 
these lines.
Given $\bvarphi \in C^\infty(\Omega, \cD(\Omega))^I \cong C^\infty(\Omega, 
\cD(\Omega)^I)$ we can form its germ at $x \in \Omega$ either 
separately for each fixed $\e$ or 
uniformly in $\e$. The first case is given by $\tau_x \circ \bvarphi \in 
(C^\infty(\unterstrich, \cD(\Omega))_x)^I$, which is determined by giving for 
each $\e$ an open neighborhood $U_\e$ of $x$ and a smoothing kernel 
$\vec\varphi_\e \in C^\infty(U_\e, \cD(\Omega))$ such that $\tau_x(\bvarphi_\e) 
= \tau_x\vec\varphi_\e$. For the second case, $\tau_x \bvarphi \in 
C^\infty( \unterstrich, \cD(\Omega)^I)_x$ denotes the germ of $\bvarphi \in 
C^\infty(\Omega, \cD(\Omega)^I)$. It is also determined by
$( \tau_x \bvarphi)_\e = \tau_x\vec\varphi_\e$ yet this time there is 
one common neighborhood $U$ of 
$x$ such that each $\vec\varphi_\e$ is an element of
$C^\infty(U, \cD(\Omega))$.

Writing the stalks with respect to either of these versions as direct 
limits we have
\[ \tau_x \circ \bvarphi \in \Bigl( \varinjlim_{U_\e \in \cU_x} C^\infty(U_\e, 
\cD(\Omega))
\Bigr)^I, \qquad \tau_x\bvarphi \in \varinjlim_{U \in \cU_x} C^\infty(U, 
\cD(\Omega)^I). \]
Note that in terms of equivalence classes $\tau_x\bvarphi$ is contained in $\tau_x \circ \bvarphi$.
Germs of type $\tau_x \circ \bvarphi$ will, however, play no role in our theory 
because locality
types incorporating such germs do not appear to be 
invariant under derivatives (cf.~\Autoref{thm_deriv}). Moreover, we will refrain 
from investigating locality types involving jets here, although one could 
consider them and look for similar results as obtained in \cite{peetre}.

Now we are ready to give the formal definition of locality types on 
$\SK(\Omega)^I \times \Omega \times I$. As motivated above, the reduction 
scheme of the full data $(\bvarphi,x,\e)$ to the ``reduced'' data on which 
$R\in\cE(\Omega)$ actually depends will be encoded by suitable mappings 
$\ell:\SK(\Omega)^I \times \Omega \times I\to X_\ell$.

\begin{definition}\label{def_loc}
 Given any mapping $\ell = \ell(\bvarphi, x, \e)$ defined on $\SK(\Omega)^I 
\times \Omega \times I$
taking values in a specific space $X_\ell$, a 
mapping $R \in \cE(\Omega)$ is called \emph{$\ell$-local} or \emph{of locality 
type $\ell$} if for all $\bvarphi, \bpsi \in \SK(\Omega)^I$, $x,y \in \Omega$ 
and $\e,\eta \in I$ the implication
\[ \ell(\bvarphi, x, \e) = \ell ( \bpsi, y, \eta) \Longrightarrow 
R(\bvarphi)_\e(x) = 
R(\bpsi)_\eta(y) \]
holds; or, equivalently, if $R$ factors through $\ell$ according to
\[
\xymatrix{
\SK(\Omega)^I \times \Omega \times I \ar[rr]^-{R} \ar[dr]_-\ell & & \bC \\
& X_\ell \ar@{.>}[ur] &
}
\]
By $\basic^\ell(\Omega)$ or $\basic[\ell](\Omega)$ we denote the 
subset of $\basic(\Omega)$ consisting of $\ell$-local elements.
\end{definition}

Only a limited number of locality types will be relevant 
for us, all of them being given by certain natural maps as, e.g., 
projections or evaluation maps. More precisely, we assume from now on
that the term giving the value $\ell(\bvarphi, x, \e)$ is a triple whose components are taken from the corresponding 
one of the following lists (i), (ii), (iii):

\begin{enumerate}[label=(\roman*)]
 \item \label{comp_1} The $\bvarphi$-component is given by any of the 
terms $\bvarphi$, $\bvarphi_\e$, $\tau_x 
\bvarphi$, $\tau_x 
\bvarphi_\e$, 
$\bvarphi(x)$, $\bvarphi_\e(x)$ or by $\star$\,. 
 \item The $x$-component is given by $x$ or by $\star$\,. 
 \item The $\e$-component is given by $\e$ or by $\star$\,.
\end{enumerate}
In the preceding lists, $\star$ denotes any fixed object whose nature 
does not matter. Regarding list (i), note that $\bvarphi 
\in \SK(\Omega)^I$, $\bvarphi_\e \in \SK(\Omega)$, $\tau_x 
\bvarphi \in \bigsqcup_y C^\infty(\unterstrich, \cD(\Omega)^I)_y$, $\tau_x 
\bvarphi_\e \in \bigsqcup_y C^\infty(\unterstrich, \cD(\Omega))_y$, 
$\bvarphi(x) \in \cD(\Omega)^I$ and $\bvarphi_\e(x) \in \cD(\Omega)$.

By $\ell_\bvarphi$, $\ell_x$ and $\ell_\e$ we denote the respective 
components of $\ell(\bvarphi, x, \e)$.
For the sake of brevity, we omit $\star$ from locality types other than 
$(\star,\star,\star)$ whenever there is no danger of confusion. 
For example, we will write $\bvarphi(x)$ for $(\bvarphi(x),\star,\star)$. The 
triple $(\star,\star,\star)$, finally, will be abbreviated simply by 
$\star$.

Formally, this gives $7\cdot 2 \cdot 2 = 28$ possible locality types.
However,for $\ell_\bvarphi \in \{ \tau_x\bvarphi, \tau_x\bvarphi_\e \}$ the 
equation $\ell(\bvarphi, x, \e) = \ell ( \bpsi, y, \eta)$ always implies $x=y$, 
hence we do not have to consider locality types involving germs but with 
the $x$-component absent; this leaves us with 24 locality types which are 
depicted in \Autoref{fig2}. Therein, we have drawn an implication arrow from a 
locality type $\ell_1$ to a locality type $\ell_2$ if locality of type $\ell_1$ 
implies locality of type $\ell_2$. In this case, $\ell_1$ is said to be 
\emph{stronger} than $\ell_2$, which we will also denote by $\ell_1 \succeq 
\ell_2$. Moreover, arrows between boxes are to be understood as a shorthand for 
the collection of arrows $\ell_1 \Rightarrow \ell_2$ where $\ell_1$ is obtained 
from $\ell_2$ by deleting one  of its components. Note that all these 
implications are strict.

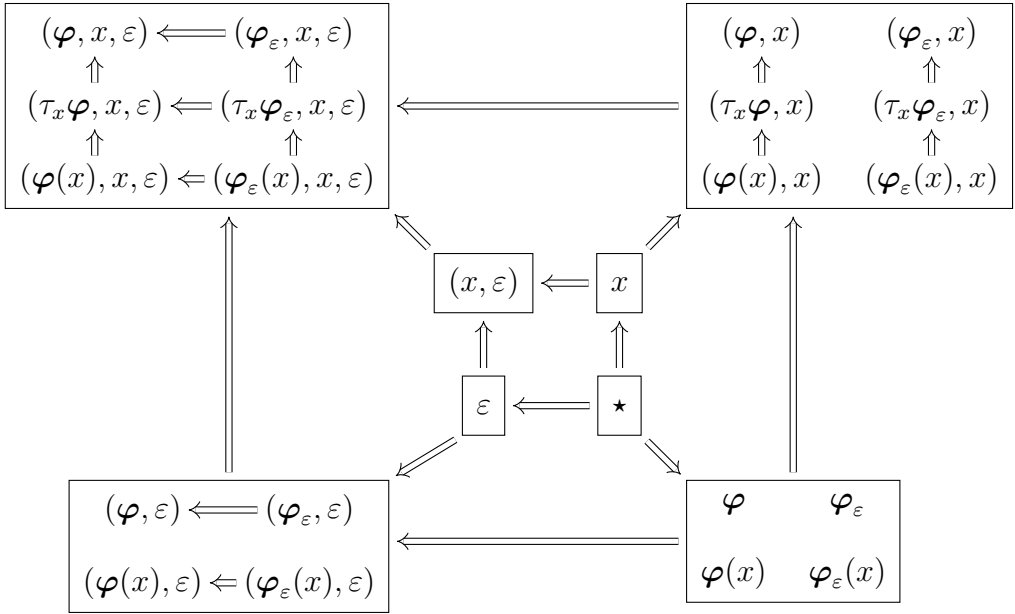
\begin{figure}[h!]
\begin{center}
\begin{tikzpicture}[
mstyle1/.style={row sep=2ex,column sep=2ex,draw,matrix of math nodes,inner sep=0.5ex},
ar2/.style={-implies,double equal sign distance,shorten >=0.1cm,shorten <=0.1cm},
ar/.style={-implies,double equal sign distance},
node distance=2em
]
\matrix[mstyle1](m5) {
(x,\e) \\
};
\matrix[mstyle1](m6) [right=of m5] {
x\vphantom( \\
};
\matrix[mstyle1](m7) [below=of m5] {
\e\vphantom( \\
};
\matrix[mstyle1](m8) at (m6 |- m7) {
\star\vphantom( \\
};
\draw[ar2] (m6.west) -- (m5.east);
\draw[ar2] (m7.north) -- (m5.south);
\draw[ar2] (m8.west) -- (m7.east);
\draw[ar2] (m8.north) -- (m6.south);

\matrix[mstyle1](mx1) [above left=of m5] {
(\bvarphi, x, \e) & ( \bvarphi_\e, x, \e) \\
    (\tau_x\bvarphi, x, \e) &  ( \tau_x \bvarphi_\e, x, \e) \\
(\bvarphi(x), x, \e) & (\bvarphi_\e(x), x, \e) \\
};

\draw[ar] (mx1-1-2) -- (mx1-1-1);
\draw[ar] (mx1-2-2) -- (mx1-2-1);
\draw[ar] (mx1-3-2) -- (mx1-3-1);
\draw[ar] (mx1-3-1) -- (mx1-2-1);
\draw[ar] (mx1-2-1) -- (mx1-1-1);
\draw[ar] (mx1-3-2) -- (mx1-2-2);
\draw[ar] (mx1-2-2) -- (mx1-1-2);

\draw[ar2] (m5.north west) -- (mx1.south east);

\matrix[mstyle1](mx4) [below right=of m8] {
    \bvarphi & \bvarphi_\e \\
\bvarphi(x) & \bvarphi_\e(x) \\
};

\draw[ar2] (m8.south east) -- (mx4.north west);

\matrix[mstyle1,anchor=south west](mx2) at (mx1.south-|mx4.west) {
    (\bvarphi, x) & ( \bvarphi_\e, x) \\
    (\tau_x\bvarphi, x) &  ( \tau_x \bvarphi_\e, x) \\
(\bvarphi(x), x) & (\bvarphi_\e(x), x) \\
};

\draw[ar] (mx2-3-1) -- (mx2-2-1);
\draw[ar] (mx2-2-1) -- (mx2-1-1);
\draw[ar] (mx2-3-2) -- (mx2-2-2);
\draw[ar] (mx2-2-2) -- (mx2-1-2);

\draw[ar2] (m6.north east) -- (mx2.south west);

\matrix[mstyle1,anchor=north east](mx3) at (mx1.east|-mx4.north) {
(\bvarphi, \e) & ( \bvarphi_\e, \e) \\
(\bvarphi(x), \e) & (\bvarphi_\e(x), \e) \\
};

\draw[ar] (mx3-1-2) -- (mx3-1-1);
\draw[ar] (mx3-2-2) -- (mx3-2-1);

\draw[ar2] (m7.south west) -- (mx3.north east);

\draw[ar2] (mx2.west) -- (mx1.east);
\draw[ar2] (mx3.north) -- (mx3.north|-mx1.south);
\draw[ar2] (mx4.west) -- (mx4.west-|mx3.east);
\draw[ar2] (mx4.north) -- (mx4.north|-mx2.south);
\end{tikzpicture}
\caption{Locality types}
\label{fig2}
\end{center}
\end{figure}

Clearly, $\basic[\ell](\Omega)$ is an algebra for every $\ell$ and even a $C^\infty(\Omega)$-module if $\ell_x = x$. The following proposition gives explicit representations for those spaces not involving germs of $\bvarphi$:
\begin{proposition}\label{basiciso}
Let $\ell$ be a locality type with $\ell_\bvarphi \in \{ \bvarphi, \bvarphi_\e, \bvarphi(x), \bvarphi_\e(x), \star \}$. Then there is an algebra isomorphism
\[ \basic[\ell](\Omega) \cong C^\infty( \cF_1(\Omega), \cF_2(\Omega)) \]
which is also a $C^\infty(\Omega)$-module isomorphism if $\ell_x = x$. Here, $\cF_1(\Omega)$ is given by
\begin{alignat*}{3}
 & \SK(\Omega)^I \textrm{ if }\ell_\bvarphi = \bvarphi, \qquad && \cD(\Omega)^I \textrm{ if }\ell_\bvarphi = \bvarphi(x), \qquad && \SK(\Omega) \textrm{ if }\ell_\bvarphi = \bvarphi_\e, \\
 & \cD(\Omega) \textrm{ if } \ell_\bvarphi = \bvarphi_\e(x),
 && \{0\} \textrm{ if }\ell_\bvarphi = \star &&
\end{alignat*}
and $\cF_2(\Omega)$ is given by
\begin{alignat*}{2}
 & C^\infty(\Omega)^I \textrm{ if }(\ell_x, \ell_\e) = (x,\e), \qquad
 && C^\infty(\Omega)\textrm{ if }(\ell_x, \ell_\e) = (x,\star), \\
 & \bC^I\textrm{ if }(\ell_x, \ell_\e) = (\star,\e), \qquad
 && \bC\textrm{ if }(\ell_x, \ell_\e) = (\star,\star).
\end{alignat*}
\end{proposition}
\begin{proof}
The codomain $X_\ell$ of $\ell$ is given by
\begin{alignat*}{3}
&\textrm{(i) } && \cF_1(\Omega) \times \Omega \times I \qquad && \textrm{if }(\ell_x, \ell_\e ) = (x, \e) \\
&\textrm{(ii) } && \cF_1 (\Omega) \times \Omega \qquad && \textrm{if }(\ell_x, \ell_\e ) = (x, \star) \\
&\textrm{(iii) } &&\cF_1(\Omega) \times I \qquad && \textrm{if }(\ell_x, \ell_\e) = (\star, \e) \\
&\textrm{(iv) } && \cF_1(\Omega) \qquad && \textrm{if }(\ell_x, \ell_\e) = (\star, \star).
\end{alignat*}
For each of the five instances of $\cF_1(\Omega)$ as above there is a 
canonical map $\cF_1(\Omega)\ni a\mapsto \tilde a\in\SK(\Omega)^I$.
The respective values of $\tilde a$ on $(\e,x)$ are given by
$a(\e,x)$ (i.e., $\tilde a = a$ in this case), $a(\e)$, $a(x)$, $a$ and $0$ 
(in the order as listed in the proposition).
We now define right inverses
$r \colon X_\ell \to \SK(\Omega)^I \times \Omega \times I$ of $\ell$ by
\begin{alignat*}{2}
&\textrm{(i) } & r(a, x, \e) &= (\tilde a, x, \e) \\
&\textrm{(ii) } & r(a, x) &= ( \tilde a, x, \e_0) \\
&\textrm{(iii) } & r(a, \e) &= (\tilde a, x_0, \e) \\
&\textrm{(iv) } & r(a) &= (\tilde a, x_0, \e_0) 
\end{alignat*}
where $x_0 \in \Omega$ and $\e_0 \in I$ are arbitrary but fixed. Given $R \in 
\basic[\ell](\Omega)$, we define the mapping $\widetilde R \colon X_\ell 
\to \bC$ by $\widetilde R \coleq R \circ r$. The function 
$\widetilde R$ is smooth in $(a,x)$ in cases (i), (ii) and smooth in 
$a$ in cases (iii), (iv) because the injection $a \mapsto \tilde a$ is 
continuous, hence it defines a mapping $\widetilde R \in C^\infty( 
\cF_1(\Omega), \cF_2(\Omega) )$. The assignment $R \mapsto \widetilde R$ clearly 
is an algebra homomorphism and in case $\ell_x = x$ it is also 
$C^\infty(\Omega)$-linear.

Conversely, if $\widetilde R \in C^\infty(\cF_1(\Omega), \cF_2(\Omega))$ is 
given, we define 
$R \colon \SK(\Omega)^I \times \Omega \times I \to \bC$ by $R \coleq 
\widetilde R \circ \ell$. Then $R$ is smooth in $(\bvarphi, x)$ because 
$\ell_\bvarphi$ (hence $\ell$) is smooth, yielding $R \in \basic[\ell](\Omega)$.

The assignment $\widetilde R \mapsto R$ is an algebra homomorphism and even $C^\infty(\Omega)$-linear if $\ell_x = x$, and is inverse to the above assignment $R \mapsto \widetilde R$. Hence, we have established the desired isomorphism.
\end{proof}

We list some familiar examples of basic spaces defined by locality conditions:
\begin{itemize}
 \item $\basic[(\bvarphi_\e, x)](\Omega) \cong C^\infty( \SK(\Omega), C^\infty(\Omega)) = \cE^f(\Omega)$.
 \item $\basic[(\bvarphi_\e(x), x)](\Omega) \cong C^\infty( \cD(\Omega), C^\infty(\Omega)) = \cE^d(\Omega)$.
 \item $\basic[\bvarphi_\e(x)](\Omega) \cong C^\infty( \cD(\Omega) ) = \cE^o(\Omega)$.
 \item $\basic[(x, \e)](\Omega) \cong C^\infty(\Omega)^I = \cE^s(\Omega)$.
\end{itemize}

For the last isomorphism, we used that $C^\infty(\{0\}, E) \cong E$ for any locally convex space $E$.
These examples show how the traditional notation for the basic spaces of 
various Colombeau algebras fits into our setting. In particular, the above basic 
spaces are obtained as $\basic^\ell(\Omega)$ if one sets $\ell$ to be $f = 
(\bvarphi_\e, x)$, $d = (\bvarphi_\e(x), x)$, $o = \bvarphi_\e(x)$ or $s = 
(x,\e)$, respectively.

In passing, we note that the isomorphisms envisaged in \Autoref{subloc} can be written, according to the (obvious slim-down of the) present terminology, as 
\begin{alignat*}{3}
&\{ R \in \cE^f(\Omega)\ |\ R \textrm{ is local}\} &&  \cong \cE^f [ \tau_x\vec\varphi, x](\Omega), \\
&\{ R \in \cE^f(\Omega)\ |\ R \textrm{ is point-local} \} && \cong \cE^f [ \vec\varphi(x), x](\Omega) && \cong \cE^d(\Omega), \\
&\{ R \in \cE^f(\Omega)\ |\ R \textrm{ is point-independent} \} && \cong  \cE^f[\vec\varphi(x)] (\Omega) && \cong \cE^o(\Omega).
\end{alignat*}

\section{Operations on the basic spaces}\label{sec_op}

In this section we will examine pullbacks, derivatives as well as the presheaf property for the basic spaces $\basic[\ell](\Omega)$ with $\ell$ as specified after \Autoref{def_loc}.

Suppose we are given a diffeomorphism $\mu \colon \Omega \to \Omega'$ between open subsets $\Omega, \Omega' \subseteq \bR^n$. Clearly the space $\cL(\cD'(\Omega), C^\infty(\Omega))$ is functorial, i.e., there is an induced pullback action along this diffeomorphism given by
$(\mu^* \Phi)(u) \coleq \mu^* ( \Phi ( \mu_* u))$ for $\Phi \in \cL(\cD'(\Omega'), C^\infty(\Omega'))$ and $u \in \cD'(\Omega)$. Because the pushforward of distributions is defined by $\langle \mu_* u, \varphi \rangle \coleq \langle u, (\varphi \circ \mu) \cdot \abso{\det \D \mu} \rangle$ (where $\D \mu$ is the Jacobian of $\mu$), the smoothing kernel $\mu^*\vec\varphi$ corresponding to $\mu^* \Phi$ is obtained as
\begin{align*}
 (\mu^* \vec\varphi)(x)(y) &\coleq (\mu^* \Phi)(\delta_y)(x) = \mu^* ( \Phi(\mu_* \delta_y))(x) \\
&= \Phi ( \mu_* \delta_y)(\mu(x)) = \langle \mu_* \delta_y, \vec\varphi(\mu(x)) \rangle \\
&= \langle \delta_y, ( \vec\varphi(\mu(x)) \circ \mu) \cdot \abso{\det \D \mu} \rangle = \vec\varphi ( \mu(x)) ( \mu(y)) \cdot \abso{ \det \D \mu(y)}
\end{align*}
where $\vec\varphi$ corresponds to $\Phi$. In other words, $(\mu^* \vec\varphi)(x) = \mu^* ( \vec \varphi ( \mu(x )) )$ as was to be expected. As is usual, we set $\mu_* \coleq (\mu^{-1})^*$.

Hence, we have the following definition.
\begin{definition}\label{def_pullback}Let $\mu \colon \Omega \to \Omega'$ be a diffeomorphism between open subsets $\Omega,\Omega' \subseteq \bR^n$ and $R \in \basic(\Omega')$. Then the pullback of $R$ along $\mu$, denoted by $\mu^*R$, is defined as the element of $\basic(\Omega)$ given by
 \[ (\mu^*R )(\bvarphi)_\e \coleq \mu^* ( R ( \mu_* \bvarphi)_\e ) \]
with $\mu_* \bvarphi \coleq (\mu_* \bvarphi_\e)_\e$.
\end{definition}

This pullback preserves locality:

\begin{theorem}
Let $\Omega,\Omega'$ be open subsets of $\bR^n$ and $\mu \colon \Omega \to \Omega'$ a diffeomorphism. Then the pullback $\mu^* \colon \basic(\Omega') \to \basic(\Omega)$ preserves all locality types of \Autoref{fig2}.
\end{theorem}
\begin{proof}
Let $R \in \basic(\Omega')$ be $\ell$-local for some locality type $\ell$ and suppose that $(\bvarphi, x, \e)$ and $(\bpsi, y, \eta)$ are given such that $\ell(\bvarphi, x, \e) = \ell ( \bpsi, y, \eta)$. We have to verify that $\ell(\mu_* \bvarphi, \mu(x), \e) = \ell ( \mu_* \bpsi, \mu(y), \eta)$ holds, which is clear for the $x$- and $\e$-components. For the $\bvarphi$-component we handle each case separately.

Case $\ell_\bvarphi = \bvarphi_\e$: $\ell_\bvarphi ( \mu_* \bvarphi, \mu(x), \e) = (\mu_* \bvarphi)_\e = \mu_* ( \bvarphi_\e) = \mu_* ( \bpsi_\eta) = (\mu_* \bpsi)_\eta = \ell_\bvarphi ( \mu_*\bpsi, \mu(y), \eta).$

Case $\ell_\bvarphi = \tau_x \bvarphi$: note that $x=y$ in this case. There exists $U \in \cU_x$ such that $\bvarphi|_U = \bpsi|_U$, which implies $(\mu_* \bvarphi)|_{\mu(U)} = ( ( \mu_* \bvarphi_{\e'})|_{\mu(U)} )_{\e'} = ( ( \mu_* \circ \bvarphi_{\e'} \circ \mu^{-1} )|_{\mu(U)})_{\e'} = ( ( \mu_* \circ \bpsi_{\e'} \circ \mu^{-1})|_{\mu(U)})_{\e'} = ( ( \mu_* \bpsi_{\e'})|_{\mu(U)})_{\e'} = (\mu_* \bpsi)|_{\mu(U)}$ and hence $\tau_{\mu(x)}(\mu_*\bvarphi) = \tau_{\mu(x)} ( \mu_*\bpsi)$.

Case $\ell_\bvarphi = \tau_x \bvarphi_\e$: again, $x$ equals $y$. There exists $U \in \cU_x$ such that $\bvarphi_\e|_U = \bpsi_\eta|_U$ and hence $(\mu_* \bvarphi_\e)|_{\mu(U)} = (\mu_* \bpsi_\eta)|_{\mu(U)}$, which means that $\tau_{\mu(x)} ( \mu_*\bvarphi)_\e = \tau_{\mu(x)} ( \mu_* \bpsi)_\eta$.

Case $\ell_\bvarphi = \bvarphi(x)$: for all $\e'$ we have $(\mu_* \bvarphi)_{\e'} ( \mu(x)) = (\mu_* \circ \bvarphi_{\e'} \circ \mu^{-1})_{\e'} ( \mu(x)) = (\mu_* \circ \bpsi_{\e'} \circ \mu^{-1})_{\e'} ( \mu(y)) = (\mu_* \bpsi)_{\e'} ( \mu(y))$, hence $(\mu_* \bvarphi)(\mu(x)) = (\mu_* \bpsi)(\mu(y))$.

Case $\ell_\bvarphi = \bvarphi_\e(x)$: $(\mu_* \bvarphi)_\e ( \mu(x)) = (\mu_* \circ \bvarphi_\e \circ \mu^{-1})(\mu(x)) = (\mu_* \circ \bpsi_\eta \circ \mu^{-1})(\mu(y)) = (\mu_* \bpsi)_\eta(y)$.

The case $\ell_\bvarphi = \star$ is trivial.
\end{proof}

For any two diffeomorphisms $\mu, \nu$ which can be composed, the pullback of \Autoref{def_pullback} satisfies
\[ (\mu \circ \nu)^* = \nu^* \circ \mu^*,\quad \id^* = \id \]
and $\basic^\ell$ is in fact a functor.

Turning to derivatives, we recall from \cite{papernew,bigone} that there are two equally natural derivatives we can define on $\cE(\Omega)$, which are given as follows. First, by differentiating the pullback of $R \in \cE(\Omega)$ along the flow of a (complete) vector field $X$ on $\Omega$, we obtain the following formula for the \emph{geometric} Lie derivative along $X$:
\[ (\widehat \D_X R)(\boldsymbol \varphi)_\e \coleq - (\ud R)(\boldsymbol \varphi)( \D_X^{\SK} \boldsymbol \varphi)_\e + \D_X ( R ( \boldsymbol \varphi)_\e ) \]
where $\ud R$ denotes the differential of $R$ \cite[3.18, p.~33]{KM} and furthermore we define $(\D^{\SK}_X \boldsymbol \varphi)_\e \coleq \D_X^{\SK} \bvarphi_\e$ with $\D_X^{\SK} \vec\varphi \coleq \D_X\vec\varphi + \D_X \circ \vec\varphi$ for $\vec \varphi \in \SK(\Omega)$. Here, $\D_X$ denotes either the usual directional derivative of a (vector-valued) function or of an $n$-form, accordingly.

The other possibility is to keep $\boldsymbol \varphi$ fixed and apply the derivative \emph{componentwise} (w.r.t.~$\e$) as in
\[
 (\widetilde \D_X R)(\boldsymbol \varphi)_\e \coleq \D_X ( R(\boldsymbol \varphi)_\e).
\]

While $\widehat \D_X$ is the derivative which commutes with the embedding of distributions, $\widetilde \D_X$ is the appropriate choice to obtain a meaningful notion of covariant derivative (cf.~\cite{papernew, bigone}). 

The following theorem expresses how locality types are preserved under these derivatives.
\begin{theorem}\label{thm_deriv}
 Let $\ell$ be a locality type as in \Autoref{fig2}, $X \in C^\infty(\Omega, \bR^n)$ a vector field on $\Omega$ and $R \in \basic[\ell](\Omega)$. Then
 \begin{enumerate}[label=(\roman*)]
  \item \label{thm_deriv.1} $\widehat \D_X R \in \basic[\ell](\Omega)$,
  \item $\widetilde \D_X R \in \basic[\ell'](\Omega)$, where $(\ell'_x, \ell'_\e) = (\ell_x, \ell_\e)$ and
  \[
   \ell'_\bvarphi = \left\{
\begin{aligned}
 & \tau_x \bvarphi \qquad & &\textrm{if }\ell_\bvarphi = \bvarphi(x),\\
 & \tau_x\bvarphi_\e \qquad & &\textrm{if }\ell_\bvarphi = \bvarphi_\e(x),\\
  &\ell_\bvarphi \qquad & & \textrm{otherwise.}
\end{aligned}\right.
  \] 
 \end{enumerate}
\end{theorem}
The proof is similar to that of \cite[Theorem 25, p.~430]{papernew}, the added $\e$-variable being straightforward to handle. A similar statement as in \ref{thm_deriv.1} for the case of $\ell_\bvarphi = \tau_x \circ \bvarphi$ does not appear to be attainable because $\tau_x \circ \bvarphi = \tau_x \circ \bpsi$ does not imply $\tau_y \circ \bvarphi = \tau_y \circ \bpsi$ for $y$ in a neighborhood of $x$.

Note that both derivatives satisfy the Leibniz rule and are equal on the
space 
$\basic[(x,\e)](\Omega)$.

Finally, we define a restriction operator for $(\tau_x\bvarphi,x,\e)$-local elements of the basic space.

\begin{theorem}\label{asdfsd}Let $U \subseteq \Omega$ be open and $R \in \cE[(\tau_x\bvarphi, x, \e)](U)$. Then for each open subset $V \subseteq U$ there exists a unique map $R|_V \in \basic[(\tau_x\bvarphi, x, \e)](V)$ such that for all $W \subseteq V$ open, $\bpsi \in \SK(V)^I$ and $\bvarphi \in \SK(U)^I$ we have the implication
\begin{equation}\label{star}
\bpsi|_W = \bvarphi|_W \Longrightarrow R|_V(\bpsi)|_W = R(\bvarphi)|_W.
\end{equation}
The map $R \mapsto R|_V$ is linear and preserves all locality types $\ell$ of \Autoref{fig2} with $\ell \succeq (\tau_x\bvarphi, x, \e)$.  Moreover, if $V' \subseteq V$ is open then $(R|_V)|_{V'} = R|_{V'}$.
\end{theorem}
\begin{proof}
Let $\cF_1$ and $\cF_2$ be the sheaves of vector spaces on $U$ given by $\cF_1(X) \coleq C^\infty(X, \cD(U)^I)$ and $\cF_2(X) \coleq C^\infty(X)^I$ for $X \subseteq U$ open, with respective sheaf spaces (\'etal\'e spaces) $\widetilde \cF_1$ and $\widetilde \cF_2$ \cite[Section 1.2, p.~110]{MR0345092}. We define a mapping $\widetilde R \colon \widetilde \cF_1 \to \widetilde \cF_2$ as follows: for any germ $\tau_x\bpsi$ in $\widetilde \cF_1$ at some $x \in U$, given by a representative $\bpsi \in C^\infty(X, \cD(U)^I)$ with $X \in \cU_x$, we set
\[ \widetilde R(\tau_x\bpsi) \coleq \tau_x ( R(\bvarphi)) \]
where $\bvarphi \in \SK(U)^I$ is chosen such that $\tau_x\bvarphi = \tau_x \bpsi$. For this purpose take any $\rho \in \cD(X)$ with $\rho \equiv 1$ in an open neighborhood $W$ of $x$ and define $\bvarphi$ as $\rho \cdot \bpsi$, extended to the whole of $U$ by zero. Because $R$ is $(\tau_x\bvarphi, x, \e)$-local, $\widetilde R(\tau_x\bpsi)$ is independent of the specific choice of $\bpsi$ and $\rho$. In fact, suppose we are given another representative $\bpsi'$ of $\tau_x\bpsi$, defined on an open neighborhood $X'$ of $x$, and a function $\rho' \in \cD(X')$ which is $1$ on an open neighborhood $W'$ of $x$. Then $\tau_x(\rho\bpsi) = \tau_x(\rho' \bpsi')$ and hence $\tau_x(R(\rho \bpsi)) = \tau_x(R(\rho' \bpsi'))$. Hence, $\widetilde R$ is well-defined.

$\widetilde R$ clearly preserves the base point $x$, i.e., $\widetilde R ( (\cF_1)_x ) \subseteq (\cF_2)_x$. Moreover, it is continuous for the natural sheaf space topologies. In fact, a basis for the topology of $\widetilde \cF_2$ is given by sets of the form $\{ \tau_x \boldsymbol f : x \in X' \}$ with $X' \subseteq U$ open and $\boldsymbol f \in \cF_2(X')$, and similarly for $\widetilde\cF_1$. Suppose we are given such $\boldsymbol f$ and some germ $\tau_x\bpsi$ in $\widetilde \cF_1$ with $x \in X'$ and $\bpsi \in \cF_1(X)$ for some open neighborhood $X$ of $x$ in $U$ such that $\widetilde R ( \tau_x\bpsi) = \tau_x \boldsymbol f$. With $\rho$ and $W$ as above we have $\widetilde R ( \tau_y \bpsi) = \tau_y ( R ( \rho \cdot \bpsi))$ for all $y \in W$ and $R(\rho \cdot \bpsi) = \boldsymbol f$ in a neighborhood of $x$. Hence, $\widetilde R ( \tau_y\bpsi) = \tau_y \boldsymbol f$ for $y$ in a neighborhood of $x$, which proves continuity of $\widetilde R$.

Consequently, $\widetilde R$ defines a sheaf morphism $\cF_1 \to \cF_2$, i.e., a family of mappings $R_X \colon C^\infty(X, \cD(U)^I) \to C^\infty(X)^I$ (for $X \subseteq U$ open) which are compatible with the restrictions of $\cF_1$ and $\cF_2$  \cite[Section 1.6, p.~114]{MR0345092}. We define $R|_V$ to be the restriction of $R_V$ to $C^\infty(V, \cD(V)^I) \subseteq C^\infty(V, \cD(U)^I)$. Clearly, $R|_V$ satisfies \eqref{star}.

The mapping $R|_V \colon C^\infty(V, \cD(V))^I \to C^\infty(V)^I$ is smooth: given any $x \in V$, choose $\rho \in \cD(V)$ which is 1 on an open neighborhood $W$ of $x$. Then given $c \in C^\infty(\bR, C^\infty(V, \cD(V))^I)$, $R|_V ( c(t))|_W = R(\rho \cdot c(t))|_W$ is smooth in $t$. Because we can cover $V$ by such sets $W$, $R|_V$ is smooth \cite[3.8, p.~28]{KM}.

For uniqueness, suppose any $R' \in \basic[(\tau_x\bvarphi, x, \e)](V)$ satisfies $R'(\bpsi)|_W = R(\bvarphi)|_W$ whenever $\bpsi \in C^\infty(V, \cD(V))^I$ and $\bvarphi \in C^\infty(U, \cD(U))^I$ satisfying $\bpsi|_W = \bvarphi|_W$ are given. For any $x \in V$ choose $\rho$ and $W$ as above. Then $R'(\bpsi)|_W = R(\rho \bpsi)|_W = R|_V(\bpsi)|_W$, which gives $R' = R|_V$.

For locality of $R|_V$ suppose we are given triples $(\bvarphi, x, \e)$ and $(\bpsi, y, \eta)$ satisfying $\ell ( \bvarphi, x, \e) = \ell ( \bpsi, y, \eta)$. To establish
\[
 R|_V(\bvarphi)_\e(x) = R|_V ( \bpsi)_\eta(y)
\]
we have to show that
\begin{equation}\label{april}
\ell ( \rho_1 \bvarphi, x, \e) = \ell ( \rho_2 \bpsi, y, \eta)
\end{equation}
holds, where $\rho_1, \rho_2 \in \cD(V)$ are equal to 1 in a neighborhood of $x$ and $y$, respectively. In case $\ell_\bvarphi \in \{ \tau_x\bvarphi, \tau_x \bvarphi_\e \}$ we have $x=y$ and hence can assume that $\rho_1 = \rho_2$, in which case equation \eqref{april} holds. If $\ell_\bvarphi \in  \{ \bvarphi(x), \bvarphi_\e(x) \}$, \eqref{april} is clear, and if $\ell_\bvarphi = \star$ then it holds trivially.

Finally, we show transitivity of this restriction operator. Suppose $V' \subseteq V$ is open. Given any $x \in V'$ and $\e \in I$, choose $\rho \in \cD(V') \subseteq \cD(V)$ such that $\rho \equiv 1$ on an open neighborhood of $x$. Then for $\bvarphi \in \SK(V')^I$ we see that $(R|_V)|_{V'} ( \bvarphi)_\e(x) = R|_V ( \rho \bvarphi)_\e(x) = R(\rho \bvarphi)_\e(x) = R|_{V'} ( \bvarphi)_\e(x)$, which means that $(R|_V)|_{V'} = R|_{V'}$.
\end{proof}

\begin{definition}Let $U \subseteq \Omega$ be open, $\ell \succeq 
(\tau_x\bvarphi, x, \e)$ and $R \in \cE[\ell](U)$. Then for any open subset $V 
\subseteq U$, the mapping $R|_V \in \cE[\ell](V)$ is called the 
\emph{restriction of $R$ to $V$}.
\end{definition}

It is obvious that we cannot drop the assumption of $(\tau_x\bvarphi, x, 
\e)$-locality for 
the definition of this restriction map.

\section{The quotient construction}\label{sec_quotient}

In this section we will perform the quotient construction on the basic
spaces $\basic^\ell(\Omega)$ which turns $C^\infty(\Omega)$ (viewed as a subspace of $\cD'(\Omega)$) into a subalgebra, provided $\ell \preceq (\bvarphi_\e(x), x)$ or $\ell \preceq (x, \e)$ (see this section's last paragraph).

Moreover, we examine which locality type at least
needs to be assumed in order to obtain a classical theorem which states that 
whenever a representative of a generalized function is known to be moderate, its 
negligibility can be tested for without resorting to derivatives.

As far as testing is concerned, for the definition of test 
objects we use only the minimal conditions which are 
necessary to obtain a Colombeau algebra with the sheaf property.

In the following definition, $B_r(x)$ denotes the Euclidean ball of 
radius $r>0$ with center $x \in \bR^n$ and $\csn(E)$ is the set of continuous 
seminorms of a locally convex space $E$. Moreover, all spaces of linear mappings are endowed with the topology of bounded 
convergence.

\begin{definition}\label{def_testobj} Let a net of smoothing kernels $\bvarphi \in \SK(\Omega)^I$ be given and denote the corresponding net of smoothing operators by $\bPhi \in \cL ( \cD'(\Omega), C^\infty(\Omega))^I$. Then $\bvarphi$ is called a \emph{test object} on $\Omega$ if
 \begin{enumerate}[label=(\roman*)]
  \item\label{def_testobj.1} $\bPhi_\e \to \id$ in $\cL(\cD'(\Omega), \cD'(\Omega))$,
  \item\label{def_testobj.2} $\forall p \in \csn ( \cL(\cD'(\Omega), C^\infty(\Omega)) )$ $\exists N \in \bN$: $p ( \bPhi_\e ) = O (\e^{-N})$,
  \item\label{def_testobj.3} $\forall p \in \csn ( \cL(C^\infty(\Omega), C^\infty(\Omega)))$ $\forall m \in \bN$: $p( \bPhi_\e|_{C^\infty(\Omega)} - \id ) = O(\e^m)$,
  \item\label{def_testobj.4} $\forall x \in \Omega$ $\exists V \in \cU_x$ $\forall r>0$ $\exists \e_0>0$ $\forall y \in V$ $\forall \e < \e_0$: $\supp \bvarphi_\e(y) \subseteq B_r(y)$.
 \end{enumerate}
We denote the set of test objects on $\Omega$ by $S(\Omega)$. Similarly, $\bvarphi$ is called a $0$-test object if it satisfies these conditions with \ref{def_testobj.1} and \ref{def_testobj.3} replaced by the following conditions:
\begin{enumerate}[label=(\roman*')]
 \item\label{def_testobj.1p}  $\bPhi_\e \to 0$ in $\cL(\cD'(\Omega), \cD'(\Omega))$,\addtocounter{enumi}{1}
 \item\label{def_testobj.3p} $\forall p \in \csn ( \cL(C^\infty(\Omega), C^\infty(\Omega)))$ $\forall m \in \bN$: $p( \bPhi_\e|_{C^\infty(\Omega)} ) = O(\e^m)$.
\end{enumerate}
The set of all $0$-test objects on $\Omega$ is denoted by $S^0(\Omega)$.

We say that a subset of $S(\Omega)$ or of $S^0(\Omega)$ is \emph{uniform} if the respective conditions hold uniformly for all its elements.
\end{definition}

Clearly, $S^0(\Omega)$ is a vector space and $S(\Omega)$ an affine space over $S^0(\Omega)$. Defining the product of $\bvarphi \in S^0(\Omega)$ and $f \in C^\infty(\Omega)$ by $(f \bvarphi)_\e(x) \coleq f(x) \cdot \bvarphi_\e(x)$, $S^0(\Omega)$ becomes a $C^\infty(\Omega)$-module.
We have the following definition of moderateness and negligibility for elements of $\basic(\Omega)$, where $\bN_0 \coleq \bN \cup \{ 0 \}$.

\begin{definition}\label{def_modneg}An element $R \in \basic(\Omega)$ is called 
\emph{moderate} if
\begin{gather*}
\forall k \in \bN_0\ \forall p \in \csn ( C^\infty(\Omega))\ \forall \bvarphi \in S(\Omega)\ \forall \bpsi_1,\dotsc,\bpsi_k \in S^0(\Omega)\ \exists N \in \bN:\\
p \bigl( \ud^k R(\bvarphi) (\bpsi_1, \dotsc, \bpsi_k)_\e\bigr) = O(\e^{-N}).
\end{gather*}
The subset of all moderate elements of $\basic(\Omega)$ is denoted by $\moderate(\Omega)$.

An element $R \in \basic(\Omega)$ is called \emph{negligible} if
\begin{gather*}
\forall k \in \bN_0\ \forall p \in \csn ( C^\infty(\Omega))\ \forall \bvarphi \in S(\Omega)\ \forall \bpsi_1,\dotsc,\bpsi_k \in S^0(\Omega)\ \forall m \in \bN:\\
p \bigl( \ud^k R(\bvarphi) (\bpsi_1, \dotsc, \bpsi_k)_\e\bigr) = O(\e^m).
\end{gather*}
The subset of all negligible elements of $\basic(\Omega)$ is denoted by $\negligible(\Omega)$.

By $\moderate^\ell(\Omega)$ and $\negligible^\ell(\Omega)$ we denote the intersection of $\moderate(\Omega)$ and $\negligible(\Omega)$ with $\basic^\ell(\Omega)$, respectively.
\end{definition}

Instead of $\moderate^\ell(\Omega)$ and $\negligible^\ell(\Omega)$ we also write $\moderate[\ell](\Omega)$ and $\negligible[\ell](\Omega)$, respectively.

There are the following canonical embeddings into subspaces of $\basic(\Omega)$:
\begin{alignat}{2}
  \sigma \colon C^\infty(\Omega) & \to \cE[x](\Omega)
               \hphantom{\bvarphi_\e()}  
                              \cong C^\infty(\Omega),
      \quad & \sigma(f)(\bvarphi)_\e(x) & := f(x), \label{embed_sigma} \\
  \iota \colon\hphantom{^\infty} \cD'(\Omega) & \to \cE[\bvarphi_\e(x)](\Omega) 
\cong         C^\infty(\cD(\Omega)),
\quad & \iota\kern1pt(u)\kern1pt(\bvarphi)_\e(x) &:= \langle u, \bvarphi_\e(x) 
\rangle. \label{embed_iota}
\end{alignat}
Furthermore, for any fixed test object $\btheta \in S(\Omega)$ one can define an embedding
\[
 \iota_\btheta \colon \cD'(\Omega) \to \cE[ (x, \e) ] (\Omega) \cong C^\infty ( 
\Omega )^I, 
\quad \iota_\btheta (u)_\e(x) \coleq \langle u, \btheta_\e(x) \rangle.
\]
This is the embedding used for special algebras.

We now prove (cf.~\cite[Theorem 6.2, p.~203]{bigone}) that in the above definition of moderateness and negligibility, if $R$ is $(\bvarphi_\e, x, \e)$-local the estimates can equivalently be taken uniformly for $\bvarphi$ and $\bpsi_1,\dotsc,\bpsi_k$ in uniform subsets of $S(\Omega)$ and $S^0(\Omega)$, respectively:

\begin{lemma}\label{lem_testuniform}
Let $R \in \basic(\Omega)$ be $(\bvarphi_\e, x, \e)$-local.
 \begin{enumerate}[label=(\roman*)]
  \item The mapping $R$ is moderate if and only if
\begin{gather*}
\forall k \in \bN_0\ \forall p \in \csn ( C^\infty(\Omega))\ \forall A \subseteq S(\Omega), B \subseteq S^0(\Omega) \textrm{ uniform}\ \exists N \in \bN: \\
  p \bigl( \ud^k R(\bvarphi)(\bpsi_1, \dotsc, \bpsi_k)_\e\bigr) = O(\e^{-N})
\end{gather*}
uniformly for $\bvarphi \in A$ and $\bpsi_1, \dotsc, \bpsi_k \in B$.
  \item The mapping $R$ is negligible if and only if
\begin{gather*}
\forall k \in \bN_0\ \forall p \in \csn ( C^\infty(\Omega))\ \forall A \subseteq S(\Omega), B \subseteq S^0(\Omega)\textrm{ uniform}\ \forall m \in \bN: \\
p \bigl( \ud^k R(\bvarphi)(\bpsi_1, \dotsc, \bpsi_k)_\e\bigr) = O(\e^m)
\end{gather*}
  uniformly for $\bvarphi \in A$, $\bpsi_1, \dotsc, \bpsi_k \in B$.
 \end{enumerate}
\end{lemma}
\begin{proof}
 Assume that $R$ is moderate but the condition in (i) does not hold. This means that there exist $k \in \bN_0$, $p \in \csn ( C^\infty (\Omega))$ and uniform subsets $A \subseteq S(\Omega)$ and $B \subseteq S^0(\Omega)$ such that
\begin{gather*}
 \forall N \in \bN\ \forall C >0\ \forall \e_0 \in I\ \exists \eta < \e_0\ \exists \bvarphi \in A\ \exists \bpsi_1, \dotsc, \bpsi_k \in B:\\
p(\ud^kR(\bvarphi)(\bpsi_1, \dotsc, \bpsi_k)_\eta) > C \cdot \eta^{-N}.
\end{gather*}
 From this we obtain a strictly decreasing sequence $(\e_n)_n \searrow 0$ and sequences $(\bvarphi^n)_n$ in $A$ and $(\bpsi_1^n)_n$, $\dotsc$, $(\bpsi^n_k)_n$ in $B$ such that
 \begin{equation}\label{dezemberalpha}
  p\bigl( \ud^k R(\bvarphi^n)(\bpsi^n_1, \dotsc, \bpsi^n_k)_{\e_n} \bigr) > n \cdot \e_n^{-n} \qquad (n \in \bN).
 \end{equation}
Choose any $\widetilde\bvarphi \in S(\Omega)$ and $\widetilde \bpsi_1, \dotsc, \widetilde \bpsi_k \in S^0(\Omega)$ such that $\widetilde \bvarphi_{\e_n} = \bvarphi^n_{\e_n}$ and $\widetilde \bpsi_{i, \e_n} = \bpsi^n_{i, \e_n}$ for all $n$ and $i=1 \dotsc k$, which is possible because $A$ and $B$ are uniform. By assumption then there is $N \in \bN$ such that
\[ \exists C>0\ \exists \e_0>0\ \forall \e<\e_0: p\bigl(\ud^kR(\widetilde \bvarphi)(\widetilde \bpsi_1, \dotsc, \widetilde \bpsi_k)_\e\bigr) \le C \e^{-N}. \]
Now for $n \ge \max(C,N)$ and $\e = \e_n < \e_0$ this gives a contradiction to \eqref{dezemberalpha}, hence the condition in (i) must hold. The case of negligibility is shown the same way.
\end{proof}

The role of assuming $(\bvarphi_\e, x, \e)$-locality in the preceding lemma is the following: to exploit the totality of conditions \eqref{dezemberalpha} by means of the composite test objects $\widetilde \bvarphi$ and $\widetilde \bpsi_1,\dotsc, \widetilde \bpsi_k$, we need that
\[ \ud^k R(\bvarphi^n)(\bpsi^n_1, \dotsc, \bpsi^n_k)_{\e_n} = \ud^k R (\widetilde \bvarphi) ( \widetilde \bpsi_1, \dotsc, \widetilde \bpsi_k)_{\e_n} \]
for all $n$; this relation, however, just results as a consequence of $(\bvarphi_\e, x, \e)$-locality.

The use of uniform sets of (0-)test objects allows one to prove the following important result, which by now is classical in all Colombeau-type algebras of generalized functions:

\begin{theorem}\label{thm123}Let $R \in \moderate(\Omega)$ be $(\bvarphi_\e, x, 
\e)$-local. 
Then $R$ is negligible if and only if
\[ \forall K \subseteq \Omega\textrm{ compact } \forall \bvarphi \in S(\Omega)\ \forall m \in \bN: \sup_{x \in K} \abso{R(\bvarphi)_\e(x) } = O(\e^m). \]
\end{theorem}

\begin{proof}
Assume $R$ satisfies the negligibility test of \Autoref{def_modneg} for $k = k_0 \in \bN_0$. Testing for negligibility of $R$ with $k = k_0 + 1$, fix $p \in \csn ( C^\infty(\Omega))$, $\bvarphi \in S(\Omega)$, $\bpsi_1, \dotsc, \bpsi_{k_0+1} \in S^0(\Omega)$ and $m \in \bN$. In order to estimate
\[ p \bigl( \ud^{k_0+1} R ( \bvarphi ) ( \bpsi_1, \dotsc, \bpsi_{k_0+1})_\e\bigr) \]
we use the fact that for any $\bpsi \in S^0(\Omega)$ and $\e \in I$, by Taylor's formula \cite[5.12, p.~59]{KM} the expression $\ud^{k_0}R ( \bvarphi + \bpsi)(\bpsi_1, \dotsc, \bpsi_{k_0} )_\e$ equals
\begin{gather*}
\ud^{k_0} R(\bvarphi)(\bpsi_1, \dotsc, \bpsi_{k_0})_\e + \ud^{k_0+1} R(\bvarphi)(\bpsi_1, \dotsc, \bpsi_{k_0}, \bpsi)_\e \\
+ \int_0^1 (1-t)\ud^{k_0 + 2} R(\bvarphi + t \bpsi) (\bpsi_1, \dotsc, \bpsi_{k_0}, \bpsi, \bpsi)_\e\,\ud t.
\end{gather*}
Because $\{ \bvarphi + s \bpsi_{k_0+1}\ |\ s \in [0,1] \}$ is a uniform set of test objects, by moderateness of $R$ there exists $N \in \bN$ such that
\[ \sup_{s \in [0,1]} p \bigl( \ud^{k_0+2} R ( \bvarphi + s \bpsi_{k_0+1})(\bpsi_1, \dotsc, \bpsi_{k_0}, \bpsi_{k_0+1}, \bpsi_{k_0+1})_\e\bigr) = O(\e^{-N}). \]
Moreover, by assumption we have
\[ \sup_{s \in [0,1]} p \bigl( \ud^{k_0} R ( \bvarphi + s \bpsi_{k_0+1})(\bpsi_1, \dotsc, \bpsi_{k_0})_\e\bigr) = O(\e^{2m+N}). \]
Setting $\bpsi = \e^{m+N}\bpsi_{k_0+1}$ above, we see that for any $\e \in (0,1]$ we have
\begin{gather*}
 \ud^{k_0+1}R(\bvarphi)(\bpsi_1, \dotsc, \bpsi_{k_0}, \bpsi_{k_0+1})_\e = \\
\e^{-m-N} \bigl( \ud^{k_0} R ( \bvarphi + \e^{m+N} \bpsi_{k_0+1})(\bpsi_1, \dotsc, \bpsi_{k_0})_\e + \ud^{k_0}R(\bvarphi)(\bpsi_1, \dotsc, \bpsi_{k_0})_\e\bigr) \\
- \e^{m+N} \int_0^1 \ud^{k_0+2} R ( \bvarphi + t \e^{m+N} \bpsi_{k_0+1})( \bpsi_1, \dotsc, \bpsi_{k_0}, \bpsi_{k_0+1}, \bpsi_{k_0+1})_\e\,\ud t.
\end{gather*}
From this the estimate
\[ p \bigl( \ud^{k_0+1}R(\bvarphi)(\bpsi_1, \dotsc, \bpsi_{k_0+1})_\e\bigr) = O(\e^m) \]
follows. Inductively, we see that $R$ is negligible already if the test of \Autoref{def_modneg} holds for $k=0$.

Our claim hence is established if we can show the estimate
\begin{equation}\label{equazione}
\forall K \csub \Omega\ \forall \alpha \in \bN_0^n\ \forall \bvarphi \in S(\Omega)\ \forall m \in \bN\ : \sup_{x \in K} \abso{ \pd_x^\alpha ( R ( \bvarphi )_\e )(x)} = O (\e^m).
\end{equation}

Similarly to the above, assume that $R$ satisfies \eqref{equazione} for $\alpha = \alpha_0 \in \bN_0^n$.
Fix $K$, $\bvarphi$ and $m$. This, together with moderateness of $R$, yields the existence of $N \in \bN$ such that both
\[ \sup_{x \in L} \abso{\pd_x^{\alpha_0 + 2 e_i} ( R ( \bvarphi)_\e)(x) } = O(\e^{-N}) \]
and
\[ \sup_{x \in L} \abso{ \pd_x^{\alpha_0} ( R ( \bvarphi)_\e)(x)} = O(\e^{2 m + N} ) \]
hold, where $e_i$ is the $i$th Euclidean basis vector of $\bR^n$ and $L$ is a compact neighborhood of $K$ with $K \subseteq L^\circ$. For small $\e$ and $x \in K$ we have the expansion
\begin{multline*}
 \pd_x^{\alpha_0} ( R( \bvarphi)_\e)(x + \e^{m + N}e_i) = \pd_x^{\alpha_0} ( R(\bvarphi)_\e) (x) + \e^{m + N} \pd_x^{\alpha_0 + e_i} ( R(\bvarphi)_\e)(x) \\
+ \e^{2m + 2N} \int_0^1 (1-t) \pd_x^{\alpha_0 + 2 e_i}(R(\bvarphi)_\e )(x + t \e^{m + N} e_i) \, \ud t.
\end{multline*}
We see that $\pd_x^{\alpha_0 + e_i} ( R(\bvarphi)_\e) = O(\e^{m})$ on $K$ follows, which gives \eqref{equazione} for all $\alpha \in \bN_0$ by induction and hence establishes the claim.
\end{proof}

The notion of association is defined as follows:
\begin{definition}
 Let $R \in \basic(\Omega)$. We say that $R$ is associated to $0$ if $\forall \bvarphi \in S(\Omega)$ we have $R(\bvarphi)_\e \to 0$ in $\cD'(\Omega)$. In this case we write $R \approx 0$.
\end{definition}
Note that every negligible element of $\basic(\Omega)$ is associated to $0$.

We finalize the quotient construction by following the scheme of \cite[Section 1.3, p.~54]{GKOS}:

\begin{enumerate}
 \item[(D1)] For each locality type $\ell$, the basic space $\basic^\ell(\Omega)$ is an associative commutative algebra with unit $1$. Moreover, there is a linear embedding $\iota \colon \cD'(\Omega) \to \basic^\ell(\Omega)$ if $\ell \preceq \bvarphi_\e(x)$ and an algebra embedding $\sigma \colon C^\infty(\Omega) \to \basic^\ell(\Omega)$ if $\ell \preceq x$. For each fixed $\btheta\in S(\Omega)$ there is a linear embedding $\iota_\btheta \colon \cD'(\Omega) \to \basic^\ell(\Omega)$ if $\ell \preceq (x,\e)$:
 \begin{align*}
  \iota \colon \cD'(\Omega) &\to \cE [ \bvarphi_\e(x)](\Omega), \\
  \iota_\btheta \colon \cD'(\Omega) &\to \cE [ ( x,\e)] (\Omega), \\
  \sigma \colon C^\infty(\Omega) &\to \cE [x] (\Omega).
 \end{align*}
 \item[(D2)] For each vector field $X \in C^\infty(\Omega, \bR^n)$ there are derivative operators $\widehat \D_X \colon \basic[\ell](\Omega) \to \basic[\ell](\Omega)$ and $\widetilde \D_X \colon \basic[\ell](\Omega) \to \basic[\ell'](\Omega)$ (with $\ell'$ as in \Autoref{thm_deriv}) which agree on $\basic[(x,\e)](\Omega)$.

On $\cE^\ell(\Omega)$, the derivative $\widehat \D_X$ extends $\D_X$ 
via $\iota$
if $\ell \preceq \bvarphi_\e(x)$ and via $\sigma$ if $\ell \preceq x$. The 
derivative $\widetilde \D_X$ extends $\D_X$ via $\iota$ on the level of 
association if $\ell \preceq \bvarphi_\e(x)$, via $\iota_\btheta$ on the level of association if $\ell \preceq 
(x,\e)$ \cite[Theorem 25 (v), p.~431]{papernew}, and via $\sigma$ if $\ell \preceq x$.
 \item[(D3)] $\moderate^\ell(\Omega)$ is a vector subspace of $\basic^\ell(\Omega)$.
 \item[(D4)] $\negligible^\ell(\Omega)$ is a vector subspace of $\basic^\ell(\Omega)$.
 \item[(D5)] There is a functorial assignment giving for each diffeomorphism $\mu \colon \Omega \to \Omega'$ a mapping $\mu^* \colon \basic^\ell (\Omega') \to \basic^\ell (\Omega)$ that extends the distributional pullback $\mu^* \colon \cD'(\Omega') \to \cD'(\Omega)$ via $\iota$ if $\ell \preceq \bvarphi_\e(x)$, via $\iota_\btheta$ on the level of association if $\ell \preceq (x,\e)$, and that extends the pullback $\mu^* \colon C^\infty(\Omega') \to C^\infty(\Omega)$ via $\sigma$ if $\ell \preceq x$. The map $\mu^*$ takes the form $\mu^* R = \mu^* \circ R \circ \mu_*$, where the last mapping $\mu_* \colon \SK(\Omega)^I \to \SK(\Omega')^I$ depends functorially on $\mu$.
\end{enumerate}
The following then is immediately seen:
\begin{enumerate}
 \item[(T1)] \begin{enumerate}[label=(\roman*)]
\item $\iota(\cD'(\Omega)) \subseteq \moderate [ \bvarphi_\e(x)](\Omega)$,\newline $\iota_\btheta ( \cD'(\Omega)) \subseteq \moderate [(x, \e)](\Omega)$.
\item $\sigma(C^\infty(\Omega)) \subseteq \moderate [x](\Omega)$.
\item $(\iota - \sigma)(C^\infty(\Omega)) \subseteq \negligible [ (\bvarphi_\e(x), x)](\Omega)$,\newline
$\iota(fg) - \iota(f)\iota(g) \in \negligible [ \bvarphi_\e(x)](\Omega)$ for all $f,g \in C^\infty(\Omega)$,\newline
$(\iota_\btheta - \sigma)(C^\infty(\Omega)) \subseteq \negligible [ (x, \e) ](\Omega)$.
\item $\iota(\cD'(\Omega)) \cap \negligible [ \bvarphi_\e(x)] ( \Omega) = \{ 0 \}$,\newline
$\iota_\btheta(\cD'(\Omega)) \cap \negligible [ (x, \e)] (\Omega) = \{ 0 \}$.
\end{enumerate}
\item[(T2)] $\moderate^\ell(\Omega)$ is a subalgebra of $\basic^\ell(\Omega)$.
\item[(T3)] $\negligible^\ell(\Omega)$ is an ideal in $\moderate^\ell(\Omega)$.
\item[(T4)] $\widehat \D_X ( \moderate^\ell(\Omega ) ) \subseteq \moderate^\ell ( \Omega)$, $\widetilde \D_X ( \moderate^\ell(\Omega ) ) \subseteq \moderate^{\ell'} ( \Omega)$.
\item[(T5)] $\widehat \D_X ( \negligible^\ell(\Omega ) ) \subseteq \negligible^\ell ( \Omega)$, $\widetilde \D_X ( \negligible^\ell(\Omega ) ) \subseteq \negligible^{\ell'} ( \Omega)$.
\item[(T6)] The set of (0)-test objects is invariant under the action induced by $\mu$.
\item[(T7)] $\mu^*$ preserves moderateness.
\item[(T8)] $\mu^*$ preserves negligibility.
\end{enumerate}
The Colombeau algebra $\quotient^\ell(\Omega)$ then is defined as follows:
\begin{enumerate}
 \item[(D6)] $\cG^\ell(\Omega) \coleq \moderate^\ell(\Omega) / \negligible^\ell(\Omega)$.
\end{enumerate}

Summarizing these properties for the quotient, we have:
\begin{itemize}
 \item $\quotient^\ell(\Omega)$ is an
 associative commutative algebra with unit, which contains $\cD'(\Omega)$ as a 
linear subspace via $\iota$ if $\ell \preceq \bvarphi_\e(x)$
and via $\iota_\btheta$ if $\ell \preceq (x,\e)$; it contains 
$C^\infty(\Omega)$ as a subalgebra via $\sigma$ if $\ell \preceq x$. 
 \item For $\ell \preceq (\bvarphi_\e(x), x)$ we have $\iota|_{C^\infty(\Omega)} = 
\sigma$, for $\ell \preceq (x,\e)$ we have $\iota_\btheta|_{C^\infty(\Omega)} = 
\sigma$.
 \item $\quotient^\ell(\Omega)$ is a differential algebra with respect to 
$\widehat \D_X$ and, if it is invariant under $\widetilde \D_X$ (as in 
\Autoref{thm_deriv}), also with respect to $\widetilde \D_X$. The derivatives 
$\widehat\D_X$ extend the derivatives from $\cD'(\Omega)$ if $\ell \preceq 
\bvarphi_\e(x)$ and from $C^\infty(\Omega)$ if $\ell \preceq x$, and the 
derivatives $\widetilde\D_X$ extend the derivatives from $\cD'(\Omega)$ on the 
level of association via $\iota$ if $\ell \preceq \bvarphi_\e(x)$ and via 
$\iota_\btheta$ if $\ell \preceq (x,\e)$, and from $C^\infty(\Omega)$ via 
$\sigma$ if $\ell \preceq x$.
\item Moreover, the construction is diffeomorphism 
invariant, i.e., the pullback $\mu^* \colon \quotient^\ell(\Omega') \to 
\quotient^\ell(\Omega)$ is well-defined; it is compatible with the classical 
pullback of distributions via $\iota$ if $\ell \preceq \bvarphi_\e(x)$, on the 
level of association via $\iota_\btheta$ if $\ell \preceq (x,\e)$, and 
compatible with the pullback of smooth functions via $\sigma$ if $\ell \preceq 
x$.
\end{itemize}

\section{The sheaf property}\label{sec_sheaf}

The first step in obtaining the sheaf property of the quotient algebra $\quotient^\ell(\Omega)$ is to show that the space of test objects or, more precisely, a quotient of it is a sheaf. The definition of the restriction map for test objects essentially rests on the following result:

\begin{lemma}\label{rhorest}For any open subset $V \subseteq \bR^n$ there exists a mapping $\rho_V \in C^\infty(V, \cD(V))$ and an open neighborhood $X$ of the diagonal in $V \times V$ such that $\rho_V(x)(y) = 1$ for all $(x,y)$ in $X$.
\end{lemma}
\begin{proof}
Cover $V$ by a locally finite family of open, relatively compact sets $V_i$ such that $\overline{V_i} \subseteq V$. For each $i$ choose $r_i>0$ such that $\overline{V_i} + \overline{B}_{r_i}(0) \subseteq V$, where $\overline{B}_{r_i}(0)$ denotes the closed Euclidean ball of radius $r_i$ centered at the origin. Let $(\chi_i)_i$ be a partition of unity subordinate to $(V_i)_i$ and fix a function $f \in C^\infty(\bR^n)$ satisfying $f(y) = 1$ for $\abso{y} \le 1/2$ and $f(y) = 0$ for $\abso{y} \ge 1$. Define
\[ \rho_V(x)(y) \coleq \sum_i \chi_i(x) \cdot f \left( \frac{y-x}{r_i} \right) \qquad (x,y \in V).
\]
It is clear that $\rho_V \in C^\infty(V, C^\infty(V))$. In order to show that $\rho_V$ is smooth into $\cD(V)$, let $x_0 \in V$ and choose an open neighborhood $W$ of $x_0$ in $V$ and a finite set $J$ such that $W \cap \supp \chi_i \ne 0$ implies $i \in J$. Then, for any $x \in W$ and $y \in \bR^n$ we have
\[ f\left( \frac{y-x}{r_i} \right) \ne 0 \Longrightarrow \abso{ \frac{y-x}{r_i} } < 1 \Longrightarrow y \in B_{r_i}(x), \]
hence $\supp \rho_V(x) \subseteq K \coleq \bigcup_{i \in J} ( \overline{V_i} + \overline{B}_{r_i}(0))$ for all $x \in W$, and $K \subseteq V$ is compact. Since $\cD_K(V)$ has the topology induced by $C^\infty(V)$, $\rho_V$ is smooth into $\cD(V)$. Moreover, if we set $r = \min_{i \in J} r_i/2$ then $\rho_V(x)(y) = 1$ for all $x \in W$ and $y \in B_r(x)$. Covering the diagonal by such sets, we obtain the claim.
\end{proof}

The following notion is useful for handling properties of test objects which are given locally for small $\e$:

\begin{definition}
\label{def_equiv}
Given open sets $U,V,W \subseteq \Omega$ with $W \subseteq U \cap V$, $\bvarphi \in \SK(U)^I$ and $\bpsi \in \SK(V)^I$ we say
\[ \bvarphi \sim_W \bpsi \mathrel{:}\Longleftrightarrow \forall x \in W\ \exists Z \in \cU_x ( W )\ \exists \e_0>0\ \forall \e<\e_0: \bvarphi_\e|_Z = \bpsi_\e|_Z, \]
where $\cU_x(W)$ is the neighborhood filter of $x$ in $W$. Instead of ``$\bvarphi \sim_W \bpsi$'' we also say ``$\bvarphi \sim \bpsi$ on $W$'', and if $W = U \cap V$ we simply write ``$\bvarphi \sim \bpsi$''.
\end{definition}

As to the the preceding definition, $\bvarphi_\e(z)$ and $\bvarphi_\e(z)$ (for $z \in Z$) have to be viewed as elements of the same space, say, of $\cD(\bR^n)$; in fact, being equal they are even elements of $\cD(U \cap V)$.

On the set of all nets of smoothing kernels on open subsets of $\Omega$ containing $W$, $\sim_W$ amounts to a relation that clearly is reflexive, symmetric and transitive. More generally, for $\bvarphi_i \in \SK(U_i)^I$ ($i=1,2,3$) we have the implication
\begin{gather*}\label{almosttrans}
\bvarphi_1 \sim \bvarphi_2\textrm{ on }U_1 \cap U_2\textrm{ and }\bvarphi_2 \sim \bvarphi_3\textrm{ on }U_2 \cap U_3 \\
\Longrightarrow \bvarphi_1 \sim \bvarphi_3\textrm{ on }U_1 \cap U_2 \cap U_3.
\end{gather*}
We denote the class of $\bvarphi \in \SK(U)^I$ in $\SK(U)^I /{\sim_U}$ by $[\bvarphi]$. Moreover, we denote by $S^1(U)$ the vector space of all $\bvarphi \in \SK(U)^I$ satisfying only \ref{def_testobj.4} of \Autoref{def_testobj}.

\begin{lemma}\label{skpresheaf}
Let $U,V$ be open subsets of $\Omega$ with $V \subseteq U$ and let $\rho_V \in C^\infty(V, \cD(V))$ be equal to 1 in a neighborhood of the diagonal in $V \times V$. Then the mapping
\[
 \begin{aligned}
r_V \colon \SK(U) &\to \SK(V)\\
\vec\varphi &\mapsto \rho_V \cdot \vec\varphi|_V
 \end{aligned}
\]
is linear and continuous. Applying it componentwise to elements $\bvarphi, \bpsi$ of $S^1(U)$, we obtain:
\begin{enumerate}[label=(\roman*)]
 \item\label{skpresheaf.5} $r_V \bvarphi \sim \bvarphi$,
 \item\label{skpresheaf.1} $r_V \bvarphi \in S^1(V)$,
 \item\label{skpresheaf.2} $\bvarphi \sim \bpsi$ implies $r_V \bvarphi \sim r_V \bpsi$,
 \item\label{skpresheaf.3} for any open subset $W \subseteq V$ we have $r_W ( r_V \bvarphi) \sim r_W \bvarphi$,
 \item\label{skpresheaf.4} the class of $r_V\bvarphi$ modulo $\sim$ does not depend on the choice of $\rho_V$.
\end{enumerate}
\end{lemma}
\begin{proof}
Linearity and continuity are clear.

For \ref{skpresheaf.5} fix $x_0 \in V$ and an open neighborhood $X$ of the diagonal in $V \times V$ such that $\rho_V(x)(y) = 1$ for all $(x,y) \in X$. As $f \colon \bR^{2n} \to \bR^{2n}$, $(x,y) \mapsto (x, y-x)$ is a homeomorphism, $f(X)$ is an open neighborhood of $(x_0,0)$ and hence contains an open set of the form $W_1 \times B_r(0) \ni (x_0, 0)$. Now $\{ (x,y) : x \in W_1, y \in B_r(x) \} = f^{-1} ( W_1 \times B_r(0)) \subseteq X$ is an open neighborhood of $(x_0, x_0)$ and $\rho_V(x)(y) = 1$ for all $x \in W_1$ and $y \in B_r(x)$. By \Autoref{def_testobj}\,\ref{def_testobj.4} there is a neighborhood $W_2$ of $x_0$ in $U$ and $\e_0>0$ such that $\supp \bvarphi_\e(x) \subseteq B_r(x)$ for $\e < \e_0$ and $x \in W_2$. Hence, for $x \in W_1 \cap W_2$ and $\e < \e_0$ we have $(r_V \bvarphi)_\e(x)(y) = \rho_V(x)(y) \cdot \bvarphi_\e(x)(y) = \bvarphi_\e(x)(y)$ for all $y$.

\ref{skpresheaf.1} follows directly from \ref{skpresheaf.5}, \ref{skpresheaf.2} is clear from $r_V \bvarphi \sim \bvarphi \sim \bpsi \sim r_V \bpsi$, \ref{skpresheaf.3} from $r_W(r_V\bvarphi) \sim r_V\bvarphi \sim \bvarphi \sim r_W \bvarphi$ and \ref{skpresheaf.4} from \ref{skpresheaf.5}.
\end{proof}

According to \Autoref{skpresheaf}, $r_V$ induces a map from $S^1(U)/{\sim_U}$ to $S^1(V)/{\sim_V}$ which -- by \ref{skpresheaf.3} -- we are entitled to view as restriction map. Thus we can write $[\bvarphi]|_V \coleq [r_V \bvarphi]$ where the specific choice of $\rho_V$ in defining $r_V$ does not matter.

Note that for $f \in C^\infty(U)$ and $\vec\varphi \in \SK(U)$ we have $r_V ( f \cdot \vec\varphi) = f|_V \cdot r_V \vec\varphi$. This means that $U \mapsto S^1(U)/{\sim_U}$ is a presheaf of $C^\infty$-modules on $\Omega$. We have even more:

\begin{proposition}\label{sksheaf}
$U \mapsto S^1(U)/{\sim_U}$ is a sheaf of $C^\infty$-modules on $\Omega$. \end{proposition}
\begin{proof}
Let $(U_i)_i$ be a family of open subsets of $\Omega$ and $U$ their union. Suppose first that we are given $\bvarphi \in S^1(U)$ such that $[\bvarphi]|_{U_i} = 0$ for all $i$. Then each $x \in U$ is contained in some $U_i$ and for any choice of $\rho_{U_i}$ we can choose a neighborhood $W_1$ of $x$ in $U_i$ and $\e_1>0$ such that $(\rho_{U_i} \bvarphi)_\e|_{W_1} = 0$ $\forall \e < \e_1$. Moreover, there is a neighborhood $W_2$ of $x$ in $U_i$ and $\e_2>0$ such that $(\rho_{U_i} \bvarphi)_\e|_{W_2} = \bvarphi_\e|_{W_2}$ for $\e < \e_2$. Hence, $\bvarphi_\e|_{W_1 \cap W_2} = 0$ for $\e < \min(\e_1, \e_2)$, which gives $\bvarphi \sim 0$.

Next, let a family $(\bvarphi^i)_i$ with $\bvarphi^i \in S^1(U_i)$ be given such that $[\bvarphi^i]|_{{U_i \cap U_j}} = [\bvarphi^j]_{U_i \cap U_j}$ for all $i,j$ with $U_i \cap U_j \ne \emptyset$. Let $(\chi_i)_i$ be a partition of unity subordinate to $(U_i)_i$ and define $\bvarphi \in \SK(U)^I$ by $\bvarphi_\e \coleq \sum_i \chi_i \cdot \bvarphi^i_\e$.

In order to see that $[\bvarphi]|_{U_i} = [\bvarphi^i]$, fix $x_0 \in U_i$ and choose an open neighborhood $W$ of $x_0$ in $U_i$ and a finite index set $J$ such that $x_0 \in U_j$ for all $j \in J$ and $\bvarphi_\e(x) = \sum_{j \in J} \chi_j(x) \bvarphi^j_\e(x)$ for $x \in W$. Because each $\bvarphi^j_\e$ equals $\bvarphi^i_\e$ in a neighborhood of $x_0$ for small $\e$, the sum equals $\bvarphi^i_\e$ in a neighborhood of $x_0$ for small $\e$, which gives the claim.
\end{proof}

The following is easily seen.
\begin{lemma}\label{lem_testobjrest}
 Let $(U_i)_i$ be a family of open subsets of $\Omega$ and $U = \bigcup_i U_i$.

If an element $\bvarphi \in S^1(U)$ satisfies any of the conditions \ref{def_testobj.1}, \ref{def_testobj.1p}, \ref{def_testobj.2}, \ref{def_testobj.3} or \ref{def_testobj.3p} of \Autoref{def_testobj} and $\bvarphi \sim \bpsi \in S^1(U)$ then also $\bpsi$ satisfies that condition. 

Moreover, $\bvarphi$ satisfies any of these conditions if and only if for each $i$, some representative of $[\bvarphi]|_{U_i}$ satisfies it.
\end{lemma}

\begin{corollary}\label{cor_testobjsheaf}
$U \mapsto S^0(U)/{\sim_U}$ is a sheaf of $C^\infty$-modules on $\Omega$ and $U \mapsto S(U)/{\sim_U}$ is a sheaf of affine spaces over $U \mapsto S^0(U)/{\sim_U}$ on $\Omega$.
\end{corollary}

\begin{lemma}\label{lem_testobj}
 Let $U,V,W$ be open sets such that $\overline{W} \subseteq U \cap V \ne \emptyset$, and let $\bvarphi \in S(V)$. Then there exists $\bpsi \in S(U)$ such that $[\bpsi]|_W = [\bvarphi]|_W$; an analogous result holds for $S^0$.
\end{lemma}
\begin{proof}
 Choose an open neighborhood $X$ of $\overline{W}$ such that $\overline{X} \subseteq U \cap V$, any $\bpsi^0 \in S(U)$ and $\chi \in C^\infty(U \cap V)$ with $\supp \chi \subseteq X$ and $\chi \equiv 1$ on $W$. Then $\chi \cdot [\bvarphi]|_{U \cap V} + (1-\chi) [ \bpsi^0 ]|_{U \cap V}$ is an element of $S(U \cap V)$ whose restriction to $(U \cap V) \cap (U \setminus \overline{X})$ equals the restriction of $\bpsi^0$ to this set, hence by \Autoref{cor_testobjsheaf} there exists $\bpsi \in S(U)$ such that
\[
 [\bpsi]|_W = ( \chi \cdot [\bvarphi]|_{U \cap V} + (1 - \chi) \cdot [ \bpsi^0]|_{U \cap V})|_W = [\bvarphi]|_W.\qedhere
\]
\end{proof}

Concerning the study of the sheaf property of the quotient, we first show that moderateness and negligibility localize. Since to this end, we have to consider restrictions $R|_{U_i}$ of $R \in \cE(U)$ to open subsets $U_i$ of $U$, \Autoref{asdfsd} sets the stage for $(\tau_x\bvarphi_\e, x, \e)$-locality to appear. Observe that for the quotient construction of \Autoref{sec_quotient} as such, no locality assumptions were needed.

\begin{theorem}\label{gnaxgu}Let $U \subseteq \Omega$ be open, $(U_i)_i$ an open cover of $U$ and $R \in \basic(U)$ be $(\tau_x\bvarphi_\e,x,\e)$-local. Then $R$ is moderate or negligible if and only if all $R|_{U_i}$ are.
\end{theorem}
\begin{proof}
In order to test $R|_{U_i}$ we have to estimate the derivatives of
\[ \ud^k R|_{U_i} ( \bvarphi ) ( \bpsi_1, \dotsc, \bpsi_k)_\e(x) \]
for $k \in \bN_0$, $\bvarphi \in S(U_i)$, $\bpsi_1, \dotsc, \bpsi_k \in S^0(U_i)$ and $x$ in a compact subset $K \subseteq U_i$. By \Autoref{lem_testobj} we can choose $\bvarphi' \in S(U)$, $\bpsi_1', \dotsc, \bpsi_k' \in S^0(U)$ such that on an open neighborhood of $K$ and for small $\e$ we have $\bvarphi_\e = \bvarphi_\e'$, $\bpsi_{1,\e} = \bpsi'_{1,\e}, \dotsc, \bpsi_{k,\e} = \bpsi'_{k,\e}$. Hence, the expression to be estimated equals $\ud^k R(\bvarphi')(\bpsi_1', \dotsc, \bpsi_k')_\e(x)$ on this neighborhood for small $\e$, so moderateness and negligibility of $R|_{U_i}$ are implied by the same property of $R$.

Conversely, in order to estimate the derivatives of $\ud^k R(\bvarphi)(\bpsi_1, \dotsc, \bpsi_k)_\e(x)$ for $k \in \bN_0$, $\bvarphi \in S(U)$, $\bpsi_1, \dotsc, \bpsi_k \in S^0(U)$ and $x$ in a compact subset $K \subseteq U$, we can assume without limitation of generality that $K \subseteq U_i$ for some $i$. Then, the expression to be estimated equals $\ud^k R|_{U_i} ( \rho_{U_i} \bvarphi ) ( \rho_{U_i} \bpsi_1, \dotsc, \rho_{U_i} \bpsi_k )_\e(x)$ (where the $\rho_{U_i}$ are as in \Autoref{rhorest}) for $x$ in an open neighborhood of $K$ and small $\e$, whence moderateness or negligibility of $R$ follows if all $R|_{U_i}$ have that property.
\end{proof}

Consequently, restriction is well-defined on the quotient if the locality type $\ell$ is strong enough.

\begin{theorem}\label{thm_sheaf}Let $\ell$ be a locality type with $\ell \succeq (\tau_x\bvarphi_\e, x, \e)$. Then $\cG^\ell$ is a sheaf of algebras on $\Omega$.
\end{theorem}

Explicitly, the condition $\ell \succeq (\tau_x\bvarphi_\e, x, \e)$ holds for all the locality types $(\tau_x\bvarphi_\e, x, \e)$, $(\bvarphi_\e(x), x, \e)$, $(\tau_x\bvarphi_\e, x)$, $(\bvarphi_\e(x), x) = d $, $(\bvarphi_\e(x), \e)$, $\bvarphi_\e(x) = 0$, $(x,\e) = s$, $x$, $\e$ and $\star$.

\begin{proof}
Let $U \subseteq \Omega$ be open and $(U_i)_i$ a covering of $U$ by open sets. That an element of $\cG^\ell(U)$ is uniquely determined by its restrictions to all subsets $U_i$ is contained in the statement of \Autoref{gnaxgu}. So, suppose that we are given functions $R_i \in \moderate^\ell (U_i)$ such that $R_i|_{U_i \cap U_j} - R_j|_{U_i \cap U_j} \in \negligible^\ell ( U_i \cap U_j )$ for all $i,j$ with $U_i \cap U_j \ne \emptyset$.

We first consider the case where $\ell_x = x$. Choose a partition of unity $(\chi_i)_i$ on $U$ subordinate to $(U_i)_i$ and for each $i$ a function $\rho_i \in C^\infty(U_i, \cD(U_i))$ as in \Autoref{rhorest} which is equal to 1 on a neighborhood of the diagonal in $U_i \times U_i$. We define the mapping $R \colon \SK(U)^I \to C^\infty(U)^I$ by
\begin{equation}\label{sheafsum}
R(\bvarphi)_\e(x) \coleq \sum_i \chi_i ( x) \cdot R_i ( \rho_i \bvarphi ) _\e (x).
\end{equation}

Clearly $R$ is smooth and $(\tau_x\bvarphi_\e,x,\e)$-local. For moderateness of $R$ we have to estimate derivatives with respect to $x$ of $\ud^k R(\bvarphi)(\bpsi_1, \dotsc, \bpsi_k)_\e(x)$ for $k \in \bN_0$, $\bvarphi \in S(U)$, $\bpsi_1, \dots, \bpsi_k \in S^0(U)$ and $x$ in a compact subset $K \subseteq U$.

For any relatively compact neighborhood of $K$ whose closure is contained in $U$ there exists a finite index set such that for $x$ in this neighborhood, the sum in \eqref{sheafsum} only has to be taken for $i$ in this index set. By the Leibniz rule it then suffices to estimate derivatives of $\ud^k R_i (\rho_i \bvarphi ) ( \rho_i \bpsi_1, \dotsc, \rho_i \bpsi_k )_\e (x)$ for $x$ in $\supp \chi_i \cap K$. But this expression has moderate growth by assumption and \Autoref{cor_testobjsheaf}.

Next, we show that $R|_{U_j} - R_j$ is negligible for all $j$. Fix $\bvarphi \in S(U_j)$ for testing and a compact set $K$ in $U_j$. Using \Autoref{lem_testobj} we choose $\bpsi \in S(U)$ such that $\bvarphi_\e = \bpsi_\e$ on a neighborhood of $K$ for small $\e$. As above, there is a finite index set such that for $x \in K$ the sum in 
\[ (R|_{U_j} - R_j)(\bvarphi)_\e(x) = \sum_i \chi_i(x) \cdot ( R_i ( \rho_i\bpsi) - R_j(\bvarphi))_\e(x) \]
runs only over this index set. Hence, by \Autoref{thm123} it suffices to estimate $R_i ( \rho_i \bpsi)_\e(x) - R_j(\bvarphi)_\e(x)$ for $x$ in the compact subset $\supp \chi_i \cap K \subseteq U_i \cap U_j$. Let $\rho \in C^\infty(U_i \cap U_j, \cD(U_i \cap U_j))$ be equal to 1 on a neighborhood of the diagonal in $(U_i \cap U_j) \times (U_i \cap U_j)$. Then for such $x$ and small $\e$,
\begin{align*}
 R_i(\rho_i \bpsi)_\e(x) - R_j(\bvarphi)_\e(x) &= R_i|_{U_i \cap U_j} ( \rho \rho_i \bpsi)_\e(x) - R_j|_{U_i \cap U_j} ( \rho \bvarphi)_\e(x) \\
 &= R_i|_{U_i \cap U_j} ( \rho \bpsi)_\e(x) - R_j|_{U_i \cap U_j} ( \rho \bpsi)_\e (x)
\end{align*}
and the claim follows by assumption.

Now $R$ has the same locality type as the $R_i$ because $\ell(\bvarphi, x, \e) = \ell ( \bpsi, y, \eta)$ (with $\ell_x = x$) implies $\ell ( \rho_i\bvarphi, x, \e) = \ell ( \rho_i \bpsi, y, \eta)$ for $x \in \supp \chi_i$.

Let us turn to the case where $\ell_x = \star$ (which means that $\ell_\bvarphi$ is either $\bvarphi_\e(x)$ or $\star$). Now we choose a partition of unity $(\chi_\alpha)_\alpha$ on $U$ such that each $\chi_\alpha$ has compact support contained in some $U_{i(\alpha)}$, and for each $\alpha$ we choose a function $q_\alpha \in \cD(U_{i(\alpha)})$ such that $q_\alpha \equiv 1$ on a neighborhood of $\supp \chi_\alpha$. Let $M \in \cL ( \cD(\bR^n), \bR^n )$ be the mapping defined by the vector-valued integral
\[ M (\varphi) \coleq \int y \cdot \varphi(y)\,\ud y \qquad (\varphi \in \cD(\bR^n)). \]

We define $R \colon \SK(U)^I \to C^\infty(U)^I$ by
\begin{equation}\label{sheafsum2}
R(\bvarphi)_\e(x) \coleq \sum_\alpha \chi_\alpha ( M ( \bvarphi_\e(x) )) \cdot R_{i(\alpha)} ( q_\alpha \bvarphi_\e(x) )_\e ( x_{i(\alpha)} ).
\end{equation}
where the $x_i \in U_i$ are arbitrary but fixed points. Clearly, $R$ is smooth and $\ell$-local. In order to show that $R$ is moderate we have to consider $x$-derivatives of the mapping
\begin{equation}\label{blubber}
\ud^k R(\bvarphi)(\bpsi_1, \dotsc, \bpsi_k)_\e(x)
\end{equation}
as above, with $x$ in a compact subset $K$ of $U$. For this we first note that the expression
\[ M(\bvarphi_\e(x) + t_1 \bpsi_{1,\e}(x) + \dotsc + t_k \bpsi_{k,\e}(x) ) \]
is uniformly bounded for $\e$ small, $x$ in a compact set and $t_1, \dotsc, t_k$ in a bounded neighborhood of $0 \in \bR$. Hence, there is a finite index set $F$ such that
\[
 R(\bvarphi')_\e(x) = \sum_{\alpha \in F} \chi_\alpha ( M (\bvarphi'_\e(x) )) \cdot R_{i(\alpha)} ( q_\alpha \bvarphi'_\e(x) )_\e ( x_{i(\alpha)})
\]
for $\bvarphi' = \bvarphi + t_1 \bpsi_1 + \dotsc + t_k \bpsi_k$, $x$ in a neighborhood of $K$, small $\e$ and $t_1, \dotsc, t_k$ close to $0$. Because
\[ \ud^k R ( \bvarphi ) ( \bpsi_1, \dotsc, \bpsi_k ) = \left.\frac{\pd}{\pd t_1}\right|_{t_1=0} \dotsm \left.\frac{\pd}{\pd t_k}\right|_{t_k=0} R ( \bvarphi + t_1 \bpsi_1 + \dotsc + t_k \bpsi_k) \]
this implies that \eqref{blubber} is given by finitely many products of terms of the form
\begin{align}
 \ud^k ( \bvarphi' & \mapsto \chi_\alpha ( M ( \bvarphi'_\e(x) ) ) ) (\bvarphi)(\bpsi_1, \dotsc, \bpsi_k) \label{blubber1} \intertext{and}
\ud^k ( \bvarphi' &\mapsto R_{i(\alpha)} ( q_\alpha \bvarphi'_\e(x) )_\e( x_{i(\alpha)} ) )(\bvarphi)(\bpsi_1, \dotsc, \bpsi_k). \label{blubber2}
\end{align}
for some new choices of $k$ and $\bpsi_1, \dotsc, \bpsi_k$. The terms of the form \eqref{blubber1} can simply be estimated by a constant which is independent of $\e$, as is easily verified. The terms of the form \eqref{blubber2} are given by
\begin{equation}
\ud^k R_{i(\alpha)} ( q_\alpha \bvarphi_\e(x)) ( q_\alpha \bpsi_{1,\e}(x), \dotsc, q_\alpha \bpsi_{k,\e}(x) )_\e ( x_{i(\alpha)} ). \label{sturm}
\end{equation}
Because
\[ \sup_{x \in K} \abso{ M(\bvarphi_\e(x))-x} \to 0 \qquad \textrm{for }\e \to 0 \]
it suffices to estimate the partial derivatives of \eqref{sturm} for $x$ in a compact neighborhood $L$ of $\supp \chi_\alpha \cap K$ such that $L \subseteq U_{i(\alpha)}$ and $q_\alpha \equiv 1$ on a neighborhood of $L$. Choose $\widetilde \bvarphi \in S(U_{i(\alpha)})$ and $\widetilde \bpsi_1, \dotsc, \widetilde \bpsi_k \in S^0(U_{i(\alpha)})$ such that $\widetilde \bvarphi_\e = q_\alpha \bvarphi_\e$, $\widetilde \bpsi_{1,\e} = q_\alpha \bpsi_{1,\e}, \dotsc, \widetilde \bpsi_{k,\e} = q_\alpha \bpsi_{k,\e}$ for small $\e$ in a neighborhood of $L$. Then, expression \eqref{sturm} is given by
\[
 \ud^k R_{i(\alpha)} ( \widetilde \bvarphi_\e(x)) ( \widetilde \bpsi_{1,\e}(x), \dotsc, \widetilde \bpsi_{k,\e}(x))_\e ( x ) = \ud^k R_{i(\alpha)} ( \widetilde \bvarphi ) ( \widetilde \bpsi_1, \dotsc, \widetilde \bpsi_k )_\e(x)
\]
from which moderateness follows.

For negligibility of $R|_{U_j} - R_j$ we proceed as before and fix $\bvarphi \in S(U_j)$ and a compact subset $K$ of $U_j$ for testing. Take $\bpsi \in S(U)$ such that $\bpsi_\e = \bvarphi_\e$ on a relatively compact neighborhood of $K$ for small $\e$. There is a finite index set such that for $x \in K$ the sum giving $(R|_{U_j} - R_j)(\bvarphi)_\e(x)$, i.e.,
\[  \sum_\alpha \chi_\alpha ( M ( \bpsi_\e(x))) \cdot ( R_{i(\alpha)} ( q_\alpha \bpsi_\e(x) ) - R_j(\bvarphi_\e(x)))_\e(  x_{i(\alpha)} ), \]
only has to be taken for $\alpha$ in this index set. It again suffices to estimate $(R_{i(\alpha)} ( q_\alpha \bpsi_\e(x)) - R_j(\bvarphi_\e(x)))_\e( x_{i(\alpha)} )$ for $x$ in compact neighborhood $L$ of $K \cap \supp \chi_\alpha$ with $L \subseteq U_{i(\alpha)} \cap U_j$, where $L$ can be chosen such that $\bpsi_\e = \bvarphi_\e$ on $L$ and $q_\alpha \equiv 1$ on a neighborhood of $L$. In this case, we have 
\begin{gather*}
(R_{i(\alpha)}(q_\alpha \bpsi_\e(x)) - R_j(\bvarphi_\e(x)))_\e(x) \\
= R_{i(\alpha)}|_{U_{i(\alpha)} \cap U_j} ( \rho \bpsi)_\e(x) - R_j|_{U_{i(\alpha)} \cap U_j} ( \rho \bpsi)_\e (x)
\end{gather*}
with $\rho \in C^\infty(U_{i(\alpha)} \cap U_j, \cD(U_{i(\alpha)} \cap U_j))$ as above, and this expression satisfies the negligibility estimates by assumption.

It is clear that also in this case $R$ has the same locality type as the $R_i$.
\end{proof}

\section{Point values}\label{sec_pv}

Although a concept of point values for Schwartz distributions was introduced by S.~\L owasiewicz in \cite{Lojasiewicz}, not every distribution needs to have a point value at every point and distributions are not uniquely determined by their point values in this sense. To the contrary, in Colombeau algebras there in fact is a concept of point values allowing for such a characterization. This concept was first introduced for the special algebra $\cG^s$ and its tempered variant \cite{ObeChar} and later extended to the full Colombeau algebras $\cG^e$ \cite{GKOS} and $\cG^d$ \cite{punktwerte}. In fact, point values can be defined in most other variants of Colombeau algebras, too (see, for example, \cite{zbMATH06289205,Eberhard,zbMATH05819591}).

In this section we are going to consider point values in the framework of the basic space $\basic(\Omega)$ in conjunction with the locality conditions of \Autoref{sec_basic}.

Beginning with a heuristic discussion, the representations of the basic spaces given by \Autoref{basiciso} at first sight suggest to define point evaluation mappings
\begin{align}
 C^\infty(\cF_1(\Omega), C^\infty(\Omega)^I) \times C^\infty(\cF_1(\Omega), \Omega^I) & \to C^\infty(\cF_1(\Omega), \bC^I) \label{eins} \\
C^\infty(\cF_1(\Omega), C^\infty(\Omega)) \times C^\infty(\cF_1(\Omega), \Omega) & \to C^\infty(\cF_1(\Omega), \bC)\label{zwei}
\end{align}
in the obvious way by componentwise application of the canonical evaluation mappings $C^\infty(\Omega)^I \times \Omega^I \to \bC^I$ and $C^\infty(\Omega) \times \Omega \to \bC$, respectively.

However, this turns out to be problematic at least in the case of the algebra $\cG^d(\Omega) = \cG[ ( \bvarphi_\e(x), x)](\Omega)$ (cf.~\cite{punktwerte} for more details). In fact, defining the point value $R(X) \in C^\infty(\cD(\Omega))$ of $R \in \cE^d(\Omega) \cong C^\infty(\cD(\Omega), C^\infty(\Omega))$ at the generalized point $X \in C^\infty(\cD(\Omega), \Omega)$ by
\begin{equation}\label{mai}
R(X)(\varphi) \coleq R(\varphi)(X(\varphi))\qquad (\varphi \in \cD(\Omega))
\end{equation}
as suggested by \eqref{zwei}, one quickly sees that this is not well-defined on the quotient. In fact, by the definition of negligibility $R$ is determined by its values $R(\bvarphi_\e(x))(x)$ for test objects $\bvarphi \in S(\Omega)$ and $x \in \Omega$, but with \eqref{mai} $R(X)$ would need to be determined by the values $R ( \bvarphi_\e(x)) ( X (\bvarphi_\e(x)))$ of which we have no information because $x \ne X(\bvarphi_\e(x))$. A workaround to this problem in the variant of $\cG^d$ of \cite{found} consists in adding the $x$-variable to representatives of generalized points and numbers and using the translation operator $T_x \colon \varphi \mapsto \varphi(.-x)$ on test functions in order to define $R(X)(\varphi,x) \coleq R ( T_{X(\varphi,x) - x} \varphi) ( X(\varphi, x))$. One succeeds in giving a point value characterization for $\cG^d$ that way \cite[Theorem 5.8]{punktwerte}. However, both adding the $x$-variable for mere technical reasons and the use of translation -- which is not available on manifolds -- point to some structural shortcomings underlying that approach.

The reasons for all these difficulties become transparent from the vantage point of $\basic(\Omega)$, i.e., for $\cF_1(\Omega) = \SK(\Omega)^I$ in \eqref{eins}, where no locality condition on $\basic(\Omega)$ is assumed and point evaluation takes the following form:
\begin{gather*}
 C^\infty(\SK(\Omega)^I, C^\infty(\Omega)^I) \times C^\infty(\SK(\Omega)^I, \Omega^I) \to C^\infty(\SK(\Omega)^I, \bC^I), \\
(R, X) \mapsto R(X),\ R(X)(\bvarphi)_\e \coleq R(\bvarphi)_\e ( X(\bvarphi)_\e).
\end{gather*}
It is a crucial fact that on the corresponding basic spaces $C^\infty(\SK(\Omega)^I, \Omega^I)$ of generalized points and $C^\infty(\SK(\Omega)^I, \bC^I)$ of generalized numbers one can introduce locality conditions analogue to those of \Autoref{def_loc} only via mappings defined on $\SK(\Omega)^I \times I$. This means that even if the elements of $\basic(\Omega)$ which are for example $(\bvarphi_\e(x), x)$-local can be represented as elements of $C^\infty ( \cD(\Omega), C^\infty(\Omega))$, there is no locality condition on the corresponding space of generalized points such that it reduces to $C^\infty(\cD(\Omega))$; this reasoning invalidates the approach given by \eqref{eins} and \eqref{zwei}. Instead, the generalized points at which elements of $C^\infty(\cD(\Omega), C^\infty(\Omega))$ should be evaluated are given by $X \in C^\infty(\SK(\Omega), \Omega)$, with $R(X)(\vec\varphi) \coleq R ( \vec\varphi(X(\vec\varphi))) ( X(\vec\varphi))$ being the correct form of point evaluation in $\cG^d$. In hindsight, the construction of \cite{punktwerte} only works because of the special form of the test objects used for $\cG^d$ in \cite{found} but cannot be transferred to manifolds directly -- our construction below, however, will also work for the manifold setting.

Summing up, while for generalized functions one can define locality mappings on $(\bvarphi, x, \e)$ as in \Autoref{def_loc}, for generalized points and numbers we can only define them on $(\bvarphi, \e)$, which means that not the same simplifications of the basic space are possible for them.

After these preliminary considerations we now give the details. First, we have generalized numbers:

\begin{definition}\label{def_gennumb}
 We define
 \begin{align*}
  \bC_M & \coleq \{\, X \in C^\infty ( \SK(\Omega)^I, \bC^I ) : \forall k \in \bN_0\ \forall \bvarphi \in S(\Omega)\\
&\qquad \forall \bpsi_1, \dotsc, \bpsi_k \in S^0(\Omega)\ \exists N \in \bN: \abso{ \ud^k X (\bvarphi)(\bpsi_1, \dotsc, \bpsi_k)_\e } = O(\e^{-N}) \, \}, \\
  \bC_\cN & \coleq \{\, X \in C^\infty ( \SK(\Omega)^I, \bC^I ) : \forall k \in \bN_0\ \forall \bvarphi \in S(\Omega)\\
&\qquad \forall \bpsi_1, \dotsc, \bpsi_k \in S^0(\Omega)\ \forall m \in \bN: \abso{ \ud^k X (\bvarphi)(\bpsi_1, \dotsc, \bpsi_k)_\e } = O(\e^m)\, \}, \\
  \widetilde \bC & \coleq \bC_M / \bC_\cN.
\end{align*}
Elements of $\widetilde \bC$ are called \emph{generalized numbers}.
\end{definition}

Note that the space of generalized numbers depends on $\Omega$ -- while this might seem slightly disconcerting at first sight it is, in fact, completely natural if one regards generalized functions as regularized distributions where the regularization procedure depends on $\Omega$; evaluating a regularized distribution at a point hence also has to incorporate this dependence in some way.

Next comes the definition of generalized points.
\begin{definition}\label{def_genpoint}
Let $\Omega$ be an open subset of $\bR^n$.
 \begin{enumerate}[label=(\roman*)]
\item \label{def_genpoint.1} By $\Omega_M$ we denote the set of all $X \in C^\infty(\SK(\Omega)^I, \Omega^I)$ such that $\forall k \in \bN_0$ $\forall \bvarphi \in S(\Omega)$, $\bpsi_1, \dotsc, \bpsi_k \in S^0(\Omega)$ $\exists N \in \bN$:
  \[ \norm{ \ud^k X(\bvarphi)(\bpsi_1, \dotsc, \bpsi_k)_\e } = O(\e^{-N}). \]
\item We introduce an equivalence relation on $\Omega_M$ by writing $X \sim Y$ if $\forall k \in \bN_0$ $\forall \bvarphi \in S(\Omega)$, $\bpsi_1, \dotsc, \bpsi_k \in S^0(\Omega)$ $\forall m \in \bN$:
  \[ \norm{ \ud^k (X-Y)(\bvarphi)(\bpsi_1, \dotsc, \bpsi_k)_\e } = O(\e^m). \]
\item \label{def_genpoint.3}We set $\widetilde\Omega \coleq \Omega_M /{\sim}$ and call its elements generalized points of $\Omega$. The set of compactly supported generalized points is the set of all $\widetilde X \in \widetilde\Omega$ which have a representative $X$ such that
\[ \exists K \subseteq \Omega\textrm{ compact}\ \forall \bvarphi \in S(\Omega)\ \exists \e_0>0\ \forall \e<\e_0: X(\bvarphi)_\e \in K, \]
and is denoted by $\widetilde \Omega_c$.
\end{enumerate}
\end{definition}

If $X \in \widetilde \Omega$ has one representative satisfying the property of \Autoref{def_genpoint} \ref{def_genpoint.3}, then it holds for every representative. Next, we define evaluation at generalized points.

\begin{definition}Let $R \in \basic(\Omega)$ and $X \in C^\infty(\SK(\Omega)^I, \Omega^I)$. The point value of $R$ at $X$ is defined as the element $R(X)$ of $C^\infty(\SK(\Omega)^I, \bC^I)$ given by
 \[ R(X)(\bvarphi)_\e \coleq R( \bvarphi)_\e (X(\bvarphi)_\e). \]
\end{definition}

Clearly, one can introduce locality conditions for generalized points and numbers analogous to the case of the basic space $\basic(\Omega)$. The details of this are completely parallel to those of \Autoref{sec_basic}. Point evaluation preserves locality as follows:

\begin{proposition}
 Let $R \in \basic(\Omega)$ be $\ell$-local for some locality type $\ell$ with $\ell_\bvarphi \in \{ \bvarphi, \bvarphi_\e, \star\}$ and $X \in C^\infty(\SK(\Omega)^I, \Omega^I)$.
 \begin{enumerate}[label=(\roman*)]
  \item If $\ell_x = \star$ then $R(X)$ is $\ell$-local.
  \item If $\ell_x = x$ and $X$ is $(\ell_\bvarphi, \ell_\e)$-local then $R(X)$ is $(\ell_\bvarphi, \ell_\e)$-local.
 \end{enumerate}
\end{proposition}
\begin{proof}
 Suppose we are given pairs $(\bvarphi, \e)$ and $(\bpsi, \eta)$ such that $\ell_\bvarphi(\bvarphi, \e) = \ell_\bvarphi ( \bpsi, \eta)$ and $\ell_\e(\bvarphi, \e) = \ell_\e ( \bpsi, \eta)$. To show that $R(X)(\bvarphi)_\e = R(X)(\bpsi)_\eta$, i.e.,
 \[ R(\bvarphi)_\e(X(\bvarphi)_\e) = R(\bpsi)_\eta ( X(\bpsi)_\eta), \]
 we need to verify that
 \[ \ell ( \bvarphi, X(\bvarphi)_\e, \e) = \ell ( \bpsi, X(\bpsi)_\eta, \eta). \]
 But this follows immediately from the assumptions in both cases.
\end{proof}

Moreover, if one knows two $(\bvarphi_\e, \e)$-local generalized numbers or points to be moderate their equivalence can be tested for without resorting to derivatives (cf.~\cite{bigone_mfval}):

\begin{proposition}
 A generalized point or generalized number which is $(\bvarphi_\e, \e)$-local is moderate or negligible if and only if the respective tests of \Cref{def_gennumb,def_genpoint} hold uniformly for $\bvarphi$ and $\bpsi_1,\dotsc,\bpsi_k$ in uniform sets of (0-)test objects, respectively.

Moreover, a generalized point or generalized number which is $(\bvarphi_\e, \e)$-local and moderate is negligible if and only if the respective tests of \Cref{def_gennumb,def_genpoint} hold for $k=0$.
\end{proposition}

\begin{proof}
 The proof is an almost verbatim copy of the proofs of \Autoref{lem_testuniform} and \Autoref{thm123}.
\end{proof}

\begin{theorem}\label{pvwelldef} Let $\widetilde R \in \quotient(\Omega)$ and $\widetilde X \in \widetilde \Omega_c$ be given. The point value $\widetilde R ( \widetilde X)$ of $\widetilde R$ at $\widetilde X$, defined as the class of $R ( X )$ in $\widetilde \bC$ for any representatives $R$ of $\widetilde R$ and $X$ of $\widetilde X$, is a well-defined element of $\widetilde \bC$.
\end{theorem}
\begin{proof}
 This can be seen by a straightforward application of the chain rule \cite[3.18, p.~33]{KM}, similarly to \cite{mfval,bigone_mfval}.
\end{proof}

Finally, we come to the characterization of generalized functions by their values at generalized points.

\begin{theorem}A generalized function $\widetilde R \in \quotient(\Omega)$ is zero if and only if one of the following conditions holds:
\begin{enumerate}[label=(\roman*)]
\item\label{pvchar1} $\widetilde R(\widetilde X) =  0$ for all generalized points $\widetilde X \in \widetilde\Omega_c[\e]$.
\item\label{pvchar2} $\widetilde R(\widetilde X) =  0$ for all generalized points $\widetilde X \in \widetilde\Omega_c[\bvarphi_\e]$.
\item\label{pvchar3} $\widetilde R(\widetilde X) =  0$ for all generalized points $\widetilde X \in \widetilde\Omega_c$.
\end{enumerate}
\end{theorem}

\begin{proof}
That $\widetilde R = 0$ implies \ref{pvchar3} was already shown in Theorem \ref{pvwelldef}, and \ref{pvchar3} trivially implies \ref{pvchar1} and \ref{pvchar2}.

In order to see that each of \ref{pvchar1} and \ref{pvchar2} implies $\widetilde R = 0$ we assume to the contrary that a representative $R$ of $\widetilde R$ is not negligible. Then by \Autoref{thm123} there exist $m \in \bN$, $\bvarphi \in S(\Omega)$, $K \subseteq \Omega$ compact, a sequence $(\e_k)_k$ with $\e_k < 1/k$ and a sequence $(x_k)_k$ in $K$ such that $\abso{ R(\bvarphi)_{\e_k} ( x_k ) } > \e_k^m$.

As in \cite{bigone_mfval} we can choose a compactly supported generalized point $X_0 \in C^\infty(\SK(\Omega), \Omega)$ which is moderate in the sense of \Autoref{def_genpoint} \ref{def_genpoint.1} and such that $X_0( \bvarphi_{\e_k} ) = x_k$ for infinitely many $k$. Defining $X \in C^\infty(\SK(\Omega)^I, \Omega^I)$ by $X(\bvarphi)_\e \coleq X_0 ( \bvarphi_\e)$, we see that $X$ is $\bvarphi_\e$-local and $R(X)(\bvarphi)_{\e_k} = R(\bvarphi)_{\e_k} ( x_k)$, hence $R(X)$ is not negligible. This means that \ref{pvchar2} entails $\widetilde R = 0$.

Alternatively, we can define $X(\bvarphi)_\e \coleq x_k$ for $\e_k \le \e < \e_{k-1}$ (with $\e_0 \coleq 1$) similar to the case of the special algebra \cite[Theorem 1.2.46, p.~38]{GKOS}. Then $X$ is $\e$-local and $R(X)$ is not negligible, because again $R(X)(\bvarphi)_{\e_k} = R(\bvarphi)_{\e_k}(x_k)$, so \ref{pvchar1} implies $\widetilde R = 0$.
\end{proof}

\section{The sharp topology}\label{sec_top}

The sharp topology on the special Colombeau algebra $\cG^s(\Omega)$ (see \cite[Definition 5, p.~44]{Biagioni} or \cite[Remark 16.3, p.~151]{MOBook}) was studied in detail by D.~Scarpal\'ezos \cite{zbMATH01876967} and adapted to the full algebra $\cG^e(\Omega)$ by J.~Aragona, R.~Fernandez and S.~O.~Juriaans \cite{zbMATH05657323}. In this section we give, for the first time, a definition of the sharp topology in the setting of full diffeomorphism-invariant algebras, i.e., on $\cG(\Omega)$ and its subspaces. This topology induces the classical sharp topology on $\cG^s(\Omega)$.

A $0$-neighborhood subbase $\cS_\Omega$ of the sharp topology on $\quotient(\Omega)$ is obtained from the definition of negligibility (\Autoref{def_modneg}) by considering, for each fixed choice of $k,p$ and $m$, the set $W_{k,p,m}$ of elements $\widetilde R \in \quotient(\Omega)$ that have a representative $R$ satisfying the growth estimate
\begin{gather*}
\forall \bvarphi \in S(\Omega)\ \forall \bpsi_1,\dotsc,\bpsi_k \in S^0(\Omega)\\ \exists \e_0, C>0\ \forall \e<\e_0: p(\ud^k R(\bvarphi)(\bpsi_1,\dotsc,\bpsi_k)_\e) \le C \e^m.
\end{gather*}
The elements of $W_{k,p,m}$ can also be specified by a condition that is independent of the representative:
\begin{gather*}
 W_{k,p,m} \coleq \{\, \widetilde R \in \quotient (\Omega)\ |\ \forall R \in \widetilde R\ \forall \bvarphi \in S(\Omega), \bpsi_1,\dotsc,\bpsi_k \in S^0(\Omega)\ \forall b>0\\
 \exists \e_0,C>0\ \forall \e<\e_0:\ p ( \ud^k R ( \bvarphi ) (\bpsi_1,\dotsc,\bpsi_k)_\e) \le C (\e^m + \e^b) \,\}.
\end{gather*}
Thus, we set
\[ \cS_\Omega \coleq \{\, W_{k,p,m}\ |\ k \in \bN_0^n,\ p \in \csn ( C^\infty(\Omega)),\ m \in \bN\,\}. \]
The set of all finite intersections of elements of $\cS_\Omega$ is a filter base on $\cG(\Omega)$ which we denote by $\cB_\Omega$.

As seen from \Autoref{algtop} below, $\cB_\Omega$ defines a topology on $\quotient(\Omega)$ and hence also on $\widetilde \bC$ via the canonical embedding $\widetilde \bC \hookrightarrow \quotient(\Omega)$. The topology of $\widetilde \bC$ can also be obtained by a filter subbase $\cS$ similar to $\cS_\Omega$ above, but with the seminorms $p$ replaced by the absolute value:
\begin{align*}
 \cS & \coleq \{\, V_{k,m}\ |\ k \in \bN_0, m \in \bN\,\}, \\
V_{k,m} & \coleq \{\, \widetilde X \in \widetilde \bC\ |\ \exists X \in \widetilde X\ \forall \bvarphi \in S(\Omega)\ \forall \bpsi_1,\dotsc,\bpsi_k \in S^0(\Omega)\\
&\qquad \exists \e_0,C>0\ \forall \e<\e_0: \abso{\ud^k X (\bvarphi)(\bpsi_1, \dotsc, \bpsi_k)_\e} \le C \e^m\, \}\\
&= \{\, \widetilde X \in \widetilde \bC\ |\ \forall X \in \widetilde X\ \forall \bvarphi \in S(\Omega)\ \forall \bpsi_1,\dotsc,\bpsi_k \in S^0(\Omega)\ \forall b>0\\
&\qquad \exists \e_0,C>0\ \forall \e<\e_0: \abso{\ud^k X (\bvarphi)(\bpsi_1, \dotsc, \bpsi_k)_\e} \le C (\e^m +\e^b)\,\}.
\end{align*}
Again, $\cS$ gives rise to a filter base $\cB$ on $\widetilde \bC$ consisting of all finite intersections of elements of $\cS$.

Our main results on these sharp topologies on $\widetilde \bC$ and $\quotient(\Omega)$ are as follows (cf.~also \cite{zbMATH05657323}):

\begin{theorem}\label{ringtop}
\begin{enumerate}[label=(\roman*)]
\item\label{ringtop.1}  There is a unique topology $\tau$ on $\widetilde \bC$ compatible with the additive group structure for which $\cB$ is a neighborhood base at zero.

\item\label{ringtop.2} The topology $\tau$ is compatible with the ring structure of $\widetilde \bC$.
\end{enumerate}
\end{theorem}
\begin{proof}
\ref{ringtop.1} follows from \cite[Ch.~III, \S 1.2, Prop.~1, p.~222]{bourbakitop} because $-V_{k,m} = V_{k,m} = V_{k,m} + V_{k,m}$ for all $k$ and $m$. \ref{ringtop.2} follows from \cite[Ch.~III, \S 6.3, p.~274]{bourbakitop} because the following conditions are satisfied:
\begin{enumerate}[label=(\arabic*)]
 \item For all $X \in \widetilde \bC$ and all $V \in \cS$ there exists $W \in \cB$ such that $X \cdot W \subseteq V$.
 \item For all $V \in \cS$ there exists $W \in \cB$ such that $W \cdot W \subseteq V$.
\end{enumerate}
This is easily verified by applying the Leibniz rule to the respective estimates.
\end{proof}

\begin{theorem}\label{algtop}
\begin{enumerate}[label=(\roman*)]
\item\label{algtop.1}  There is a unique topology $\tau_\Omega$ on 
$\cG(\Omega)$, the sharp topology,
compatible with the additive group structure for which $\cB_\Omega$ is a 
neighborhood base at zero.

\item\label{algtop.2} The topology $\tau_\Omega$ is compatible with the $\widetilde \bC$-algebra structure of $\quotient(\Omega)$.

\item\label{algtop.3} The topology $\tau_\Omega$ induces, via the canonical embedding $\widetilde \bC \hookrightarrow \cG(\Omega)$, the topology $\tau$ on $\widetilde \bC$.

\item\label{algtop.4} On the special algebra $\quotient[(x,\e)](\Omega) \subseteq \quotient(\Omega)$, $\tau_\Omega$ induces the usual sharp topology.
\end{enumerate}
\end{theorem}
\begin{proof}
Again, \ref{algtop.1} follows because $-W_{k,p,m} = W_{k,p,m} = W_{k,p,m} + W_{k,p,m}$. \ref{algtop.2} follows from \cite[Ch.~III, \S 6.6, p.~279]{bourbakitop}  because the conditions for continuity of $\widetilde \bC \times \quotient(\Omega) \to \quotient(\Omega)$,
\begin{enumerate}[label=(\arabic*)]
 \item $\forall R \in \quotient(\Omega)$ $\forall V \in \cS_\Omega$ $\exists U \in \cB$: $U \cdot R \subseteq V$,
 \item $\forall X \in \widetilde \bC$ $\forall V \in \cS_\Omega$ $\exists U \in \cB_\Omega$: $X \cdot U \subseteq V$,
 \item $\forall V \in \cS_\Omega$ $\exists W \in \cB_\Omega$ $\exists U \in \cB$: $U \cdot W \subseteq V$,
\end{enumerate}
and the conditions for continuity of $\quotient(\Omega) \times \quotient(\Omega) \to \quotient(\Omega)$,
\begin{enumerate}[label=(\arabic*)]
 \item $\forall R \in \quotient(\Omega)$ $\forall V \in \cS_\Omega$ $\exists W \in \cB_\Omega$: $R \cdot W \subseteq V$, 
\item $\forall V \in \cS_\Omega$ $\exists W \in \cB_\Omega$: $W \cdot W \subseteq V$,
\end{enumerate}
are easily verified.

\ref{algtop.3} and \ref{algtop.4} are evident from the definitions.
\end{proof}

\section{Conclusion}\label{sec_concl}

Let us summarize how the usual classical Colombeau algebras fit into the setting developed in this article:
\begin{enumerate}[label=(\roman*)]
 \item The \emph{special algebra} $\quotient^s(\Omega)$ corresponds to taking locality type $\ell = (x,\e)$ and fixing a single test object $\btheta\in S(\Omega)$ defining the embedding $\iota_\btheta$. 
 \item\label{sec_concl.2} Colombeau's original algebra $\quotient^o(\Omega)$ of \cite{ColNew} corresponds to the case $\ell = \bvarphi_\e(x)$. To recover $\cG^o(\Omega)$ in all detail would require the following (merely technical) adjustments of our construction:
\begin{enumerate}[label=(\alph*)]
\item Introducing a graded space of test objects similar to the one used in $\quotient^o(\Omega)$, i.e., defining test objects of order $q$ by demanding convergence of order $\e^q$ instead of $\e^m$ for all $m$.
\item Adapting the definitions of moderateness and negligibility accordingly.
\end{enumerate}
It is expected that the results of this article can equally be established taking into account these modifications.
 \item The elementary algebra $\quotient^e(\Omega)$ also fits into our scheme subject to the same modifications as specified in \ref{sec_concl.2}; it would then be obtained by using only test objects given by convolution with scaled mollifiers having integral one and a certain number of vanishing moments, as well as dropping smooth dependence of $R$ on $\bvarphi$ when formulated in the C-formalism (cf.~\cite[Section 2.3.2]{GKOS}).
 \item The diffeomorphism invariant algebra $\quotient^d(\Omega)$ of \cite{found} corresponds to the case $\ell = (\bvarphi_\e(x), x)$, again with a slight technical adaptation of test objects.
 \item The algebra $\quotient^f(\Omega)$ of \cite{papernew} corresponds to the case $\ell = (\bvarphi_\e, x)$.
\end{enumerate}

Our study of the quotient construction and the sheaf property shows that in practice only locality types $\ell \succeq ( \tau_x\bvarphi_\e, x, \e)$ are useful. The algebra $\quotient [ (\tau_x\bvarphi_\e, x, \e) ](\Omega)$ not only allows for all desirable properties of Colombeau algebras to be obtained, but even furnishes a true unification of the full and special settings of Colombeau algebras: while in principle the generality of full Colombeau algebras is available in it, one can always fix a test object $\btheta$ and work with the embedding $\iota_\btheta$ as in the special algebra. Moreover, one can project any generalized function to an element of the special algebra by the mapping
\[ 
 \pi_\btheta \colon \quotient [ (\tau_x\bvarphi_\e, x, \e) ](\Omega) \to \quotient [ (x, \e) ](\Omega) = \quotient^s(\Omega)
 \]
which is defined on representatives $R \in \basic [ ( \tau_x\bvarphi_\e, x, \e)](\Omega)$ by
\[ (\pi_\btheta R)_\e(x) \coleq R(\btheta)_\e(x). \]
This gives the possibility of specifying a preferred embedding (i.e., regularization procedure) for certain distributions also in the full setting, which is a useful property to have in concrete applications of the theory.

\subsection*{Acknowledgments}

E.~A.~Nigsch was supported by grants P23714 and P26859 of the Austrian Science Fund (FWF).

\end{document}